\documentclass[12pt,a4paper]{article}
\usepackage{amsmath,amsthm,amssymb, bbm, bm}
\usepackage{graphicx}

\newtheorem{theorem}{Theorem}
\newtheorem{proposition}[theorem]{Proposition}
\newtheorem{corollary}[theorem]{Corollary}
\newtheorem{lemma}[theorem]{Lemma}
\newtheorem{remark}[theorem]{Remark}

\newcommand{\cref}[1]{Corollary~\textup{\ref{#1}}}

\newcommand{\defeq}{\mathrel{\mathop:}=}

\newcommand{\R}{\mathbb{R}}
\newcommand{\C}{\mathbb{C}}
\newcommand{\N}{\mathbb{N}}
\newcommand{\E}{\mathbb{E}}
\newcommand{\V}{\mathrm{Var}}
\newcommand{\Prob}{\mathbb{P}}
\newcommand{\T}{\mathbb{T}^{n} /S_n}

\newcommand{\Real}{\mathfrak{Re}}
\newcommand{\Res}{\mathrm{Res}}
\newcommand{\Prime}{\mathcal{P}}
\newcommand{\abc}{\mathbb{T}}

\title{Gaussian Approximation of the Distribution of Strongly Repelling Particles on the Unit Circle}
\author{Alexander Soshnikov\footnote{Department of Mathematics, University of California, Davis; soshniko@math.ucdavis.edu} , Yuanyuan Xu\footnote{Department of Mathematics, University of California, Davis; yyxu@math.ucdavis.edu}}
\date{}

\begin{document}
 \maketitle

\begin{abstract}
In this paper, we consider a strongly-repelling model of $n$ ordered particles $\{e^{i \theta_j}\}_{j=0}^{n-1}$ with the density $p({\theta_0},\cdots, \theta_{n-1})=\frac{1}{Z_n} \exp \left\{-\frac{\beta}{2}\sum_{j \neq k} \sin^{-2} \left( \frac{\theta_j-\theta_k}{2}\right)\right\}$, $\beta>0$. Let $\theta_j=\frac{2 \pi j}{n}+\frac{x_j}{n^2}+const$ such that $\sum_{j=0}^{n-1}x_j=0$. Define $\zeta_n \left(  \frac{2 \pi j}{n}\right) =\frac{x_j}{\sqrt{n}}$ and extend $\zeta_n$ piecewise linearly to $[0, 2 \pi]$. We prove the functional convergence of $\zeta_n(t)$ to $\zeta(t)=\sqrt{\frac{2}{\beta}} \Real \left(  \sum_{k=1}^{\infty} \frac{1}{k} e^{ikt} Z_k  \right)$, where $Z_k$ are i.i.d. complex standard Gaussian random variables.

\end{abstract}

\section{Introduction}
The study of random matrix theory (RMT) can be traced back to sample covariance matrices studied by J. Wishart in data analysis in 1920s-1930s.  In 1951, E. Wigner associated the energy levels of heavy-nuclei atoms with Hermitian matrices whose components are i.i.d. random variables.  In 1960s, F. Dyson and M. Mehta introduced three archetypal types of matrix ensembles: Gaussian Orthogonal Ensemble (GOE), Gaussian Unitary Ensemble (GUE), and Gaussian Symplectic Ensemble (GSE), see \cite{matrixbook2}. In particular, the GOE/GUE/GSE is defined to be the ensemble of $n \times n$ Real Symmetric/Hermitian/Hermitian Quaternionic  matrices equipped with a probability measure given by
$$P^{g}_{\beta}(H)=const_n \cdot e^{-\frac{1}{2}Tr(H^2)} dH,$$
where $dH$ is the Lebesgue measure on the appropriate space of matrices.

 The joint probability density for the eigenvalues $\lambda_1, \cdots, \lambda_n \in \R$ of the GOE/GUE/GSE is given by
\begin{equation}\label{beta}
p^g_{\beta}(\lambda_1, \cdots, \lambda_n)=\frac{1}{Z^g_{n,\beta}} \prod_{j < k} |\lambda_j-\lambda_k|^{\beta}  e^{-\sum_{i=1}^n \frac{1}{2} \lambda_i^2} ,
\end{equation}
and
\begin{equation}
Z^g_{n,\beta}=(2 \pi)^{\frac{n}{2}} \prod_{j=1}^n \frac{\Gamma\left( \frac{\beta}{2}j+1 \right)}{\Gamma\left( \frac{\beta}{2}+1 \right)},
\end{equation}
where $\beta=1$ for the GOE, $\beta=2$ for the GUE, and $\beta=4$ for the GSE. The ensemble (\ref{beta}) for general $\beta>0$ is called the Gaussian $\beta$-ensemble (see \cite{D}). 

Equally important circular ensembles, namely Circular Orthogonal Ensemble (COE), Circular Unitary Ensemble (CUE), and Circular Symplectic Ensemble (CSE) were introduced in \cite{dyson1}. In particular, the CUE is defined to be the ensemble of $n \times n$ unitary matrices equipped with the Haar measure. The ordered eigenvalues are denoted as $\{e^{i \theta_j}\}_{j=0}^{n-1}$, where $0 \leq \theta_0 \leq \cdots \leq \theta_{n-1} \leq 2 \pi$. The joint probability density of the angles $\{\theta_j\}_{j=0}^{n-1}$ is given by
\begin{equation}\label{cue2}
p^c_{\beta}(\theta_1, \cdots, \theta_n)=\frac{1}{Z^c_{n ,\beta}} \prod_{0 \leq j < k \leq n-1}  \left| e^{i \theta_j} -e^{i \theta_k} \right|^{\beta}
\end{equation}
\begin{equation}\label{cue}
=\frac{1}{Z^c_{n, \beta}} \exp \left\{ \frac{\beta}{2} \sum_{ j \neq k} \log \left| 2 \sin \frac{\theta_j-\theta_k}{2} \right| \right\},
\end{equation}
with $\beta=2$ and
\begin{equation}\label{normalization}
Z^c_{n, \beta}=\frac{(2 \pi)^n}{n!} \frac{\Gamma\left(\frac{\beta n}{2}+1 \right)}{\Gamma \left( \frac{\beta}{2} +1\right)}.
\end{equation}
Similarly, the joint probability density for the COE/CSE is given by (\ref{cue2})-(\ref{normalization}) with $\beta=1$ for the COE and $\beta=4$ for the CSE. The generalized ensemble for $\beta>0$ is named the Circular $\beta$-ensemble (see \cite{F}).
These random matrix ensembles were originally introduced in Physics, but recently have played an important role in linking RMT with Number Theory, because of the connections with the Riemann zeta function.

 The Riemann zeta function is defined to be
$$\zeta(s)=\sum_{n=1}^{\infty} \frac{1}{n^s}=\prod_{p \in \Prime} \left(1-\frac{1}{p^s} \right)^{-1}$$
for $\Real s>1$ and can be analytically extended to the whole complex plane. In 1859, B. Riemann conjectured that besides the negative even integers, all other (non-trivial) zeros of $\zeta(z)$ have the form of $z=\frac{1}{2}+i \gamma_j,$ where $\gamma_j \in \R.$
Assuming the Riemann's hypothesis, in 1970s, H. Montgomery proved in \cite{mont} that under some technical conditions, the two-point correlation function of  $\gamma_j$'s  on the scale of their mean spacing is given by
\begin{equation} \label{2point}
1-\frac{\sin^2 (\pi x)}{\pi^2 x^2}.
\end{equation} 
Later, A. Odlyzko provided numerical support for Montgomery's results in \cite{num}, and Z. Rudnick and P. Sarnak extended Montgomery's results to higher order correlations in \cite{Lfunction}. F. Dyson pointed out that (\ref{2point}) coincides with the two point correlation function of the scaled eigenvalues of GUE, which is the same as that of the scaled eigenphases of CUE as $n \rightarrow \infty$. In general, the limiting distribution of the non-trivial zeros of L-functions is believed to coincide with the limiting local eigenvalue statistics in CUE/GUE model. Since then, many efforts were made to find the deep connection between zeta function and random matrices, see e.g. \cite{berry}, \cite{keating2} and \cite{keating3}.  In 2000, J. Keating and N. Snaith made an important contribution in connecting the characteristic polynomial of CUE and value distribution of $\zeta(z)$ on the critical line, see \cite{keating}. They showed that the distribution of values taken by $\log \det \left( e^{ i s}-U_n  \right)$ averaged over $U_n \in$ CUE is a good approximation to the value distribution of $\log \zeta\left(\frac{1}{2}+it\right)$  for large $n, t$ given the relation that $n=\log \frac{t}{2 \pi}(1+o(1))$.

  In 1988, K. Johansson \cite{kurt} proved the Central Limit Theorem(CLT) for the linear statistics in the Circular $\beta$-ensemble.

\begin{theorem}\label{kurt}
Let $f \in C^{1+\epsilon}(S^1)$, $\epsilon>0$. Then $\sum_{j=0}^{n-1} f(\theta_j)-\frac{n}{2 \pi} \int_{0}^{2 \pi} f(x) dx$ converges in law to the Gaussian distribution $N(0, \frac{2}{\beta}\sum_{k=-\infty}^{\infty} |k| |c_{k}|^2)$,
where $c_{k}=\frac{1}{2 \pi} \int_{0}^{2 \pi} f(x) e^{-i kx} dx$.
\end{theorem}

\begin{remark}
For $\beta=2$, the result holds under the optimal condition $\sum_{k=-\infty}^{\infty} |k| |c_{k}|^2 < \infty$ (see also \cite{diaco}, \cite{soshcue2} and references therein).
\end{remark}

To study the characteristic polynomial of CUE, one can write $ \log |\det \left( e^{ i s}-U_n \right)|=\sum_{j=1}^n \log |e^{is}-e^{i \theta_j}|$. Due to the singularity of the logarithm function, we cannot apply Theorem \ref{kurt}.  T. Baker and P. Forrester proved in \cite{F1} that $\frac{\sqrt{2} \log |\det \left( e^{ i s}-U_n  \right)|}{\sqrt{ \log n}}$ converges in distribution to standard normal distribution for fixed $s$.  It was proved in \cite{cue1} that for the CUE $(\beta=2)$, $\sqrt{2} \log |\det \left( e^{ i s}-U_n  \right)|$ converges in distribution to a generalized random function
\begin{equation}\label{T}
T(s)=\Real \left( \sum_{k=1}^{\infty} \frac{e^{i k s}}{\sqrt{k}}   Z_k \right),
\end{equation}
where $Z_k$ are i.i.d complex standard  Gaussian variables (see also \cite{F1} and \cite{DE}).

 The generalized random function $T(s)$ makes another appearance in the Circular $\beta-$ensemble as follows. One can show that the joint probability density of (\ref{cue}) obtains its maximum at the lattice configuration $\theta_j=\frac{2 \pi j}{n}+const~(0 \leq j \leq n-1).$
 Write
$$\theta_j=\frac{2 \pi j}{n}+\frac{t_j}{n}+const.$$ 
Let us choose the constant such that $\sum_{j=0}^{n-1} t_j=0$ and take Taylor expansion of (\ref{cue}) around this critical configuration. If we ignore the cubic and higher terms, then we get, as an approximation, a multivariate Gaussian distribution on the hyperplane $\sum_{j=0}^{n-1} t_j=0$ with the density
\begin{equation}\label{newcue}
\tilde{p}_{g}(t)=\frac{1}{\tilde{Z}_g} \exp \left\{ -\frac{\beta}{16} \sum_{j \neq k}  \frac{1}{\sin^2 \left( \frac{\pi(j-k)}{n}  \right)} \frac{(t_j-t_k)^2}{n^2} \right\}.
\end{equation}
It can be shown that $t_j$ from (\ref{newcue}) can be expressed as
\begin{equation}
t_j=\frac{2}{\sqrt{\beta}} \Real \left( \sum_{k=1}^{n} \frac{e^{ \frac{2 \pi i j k}{n}}}{\sqrt{k}}   Z_k \right)(1+\epsilon_n),
\end{equation}
where $\epsilon_n$ is a negligible random error term with $\V(\epsilon_n)=o_n(1)$. Moreover, the linear statistics $\sum_{j=0}^{n-1} f\left( \frac{2 \pi j}{n}+\frac{t_j}{n} \right)$ satisfies the same CLT as in Theorem \ref{kurt}.

\begin{remark}
We refer the reader to Section 3.4 of \cite{matrixbook} for a related discussion of the mesoscopic structure of GUE eigenvalues.
\end{remark}

 This indicates that the generalized random function $T(s)$ defined in (\ref{T}) gives a good approximation of the eigenvalue statistics of CUE. However, it is not entirely clear in what sense we can ignore cubic and higher order terms of the Taylor expansion of (\ref{cue}). This motivated us to consider a new model of interacting particles on the unit circle with stronger repulsion than that in the Circular $\beta-$ensembles. The purpose of this paper is to establish the Gaussian approximation for the distribution of strongly repelling particles on the unit circle.

 Throughout this paper, the letters $C_k$, $C'_k$, $c_k$ and $c'_k$$(k \in \N)$ denote positive constants whose values might change in different parts of the paper, but are always independent of $n$. We say $a_n \ll b_n$ or $a_n=o(b_n)$ if $\frac{a_n}{b_n} \rightarrow 0$ as $n \rightarrow \infty$ and $a_n=O(b_n)$ if there exists some positive constant $C$ such that $|a_n| \leq C |b_n|$ as $n \rightarrow \infty$. If $a_n \rightarrow 0$ as $n \rightarrow \infty$ and the decay rate does not depend on other parameters, we say $a_n=o_n(1)$.  Also, we denote $a_n \sim b_n$ if there exist positive constants c and C such that $c b_n \leq X \leq C a_n$ as $n \rightarrow \infty$.

\section{Set up and notations}
 Consider a strong repulsion model of particles distributed on $$\T=\left\{ \theta=(\theta_0, \cdots, \theta_{n-1}) \in [0, 2 \pi]^{n} ~| ~\theta_0 \leq \theta_1 \leq \cdots \leq \theta_{n-1} \right\}.$$ 
The joint probability density is defined as
\begin{equation} \label{ptheta}
q({\theta})=\frac{1}{Z_n} e^{H_{n,\beta}({\theta})},
\end{equation}
where 
\begin{equation} \label{}
H_{n,\beta}({\theta})=-\frac{\beta}{2} \sum_{i \neq j} \frac{1}{\sin^2 \left( \frac{\theta_i-\theta_j}{2}\right)},
\end{equation}
and 
\begin{equation} \label{}
Z_n=\int_{\T} e^{H_{n, \beta}({\theta})} d \theta .
\end{equation}

 For any measurable subset $A \in \T$, let $\Prob(A)=\int_A q(\theta) d \theta$.

 Note that the repulsion in $H_{n,\beta}(\theta)$  is stronger than the logarithmic one in (\ref{cue}).

 Let 
\begin{equation}\label{theta}
\theta_i=\frac{2 \pi i}{n}+\psi+\frac{x_i}{n^2},
\end{equation} where $\psi$ is a constant chosen so that 
\begin{equation}
\sum_{i=0}^{n-1} x_i=0.
\end{equation}

Thus, 
\begin{equation}\label{psi}
\psi=\frac{1}{n} \sum_{i=0}^{n-1} \theta_i-\frac{\pi(n-1)}{n}.
\end{equation}

For notational simplicity, define $\alpha_i=\frac{2 \pi i}{n}+\psi, \alpha=(\alpha_0, \cdots, \alpha_{n-1})$, $x=(x_0, \cdots, x_{n-1})$, then $\theta=\alpha+\frac{x}{n^2}$.

 Next, we introduce some useful lemmas.

\begin{lemma}\label{integral}
The probability density $q(\theta)$ in (\ref{ptheta}) obtains its maximum at $\theta=\alpha$, and 
\begin{equation} \label{}
H_{n,\beta}(\alpha)-H_{n,\beta}(\theta)=\frac{\beta}{2}\sum_{i \neq j}\frac{(x_i-x_j)^2}{n^4} \int_0^1 \frac{\frac{1}{2}+\cos^2 \left( \frac{\pi(i-j)}{n} +\tau \frac{x_i-x_j}{2n^2} \right)}{ \sin^4 \left( \frac{\pi(i-j)}{n} +\tau \frac{x_i-x_j}{2n^2} \right)} \cdot (1- \tau) d \tau.
\end{equation}
\end{lemma}

This implies that the lattice configuration $\theta_i=\frac{ 2 \pi j}{n}+const$ is the ground state of the strongly repelling particle system.

\begin{lemma}\label{lemma}
The following identities hold.
\begin{equation} \label{formula1}
\sum_{k=1}^{n-1} \frac{1}{ \sin^2 \left( \frac{ \pi k}{n} \right)}=\frac{n^2-1}{3}.
\end{equation}

\begin{equation} \label{formula2}
\sum_{k=1}^{n-1} \frac{1}{ \sin^4 \left( \frac{ \pi k}{n} \right)}=\frac{(n^2-1)(n^2+11)}{45}.
\end{equation}

\begin{equation} \label{formula3}
\sum_{k=1}^{n-1} \frac{\sin^2 \left(m  \frac{ \pi k  }{n}  \right)}{ \sin^2 \left( \frac{ \pi k}{n} \right)}=m (n-m) ~~(1 \leq m \leq n-1).
\end{equation}

\begin{equation} \label{formula4}
\sum_{k=1}^{n-1} \frac{\sin^2 \left(m  \frac{ \pi k  }{n}  \right)}{ \sin^4 \left( \frac{ \pi k}{n} \right)}=\frac{m^2(n-m)^2}{3}+\frac{2}{3} m (n-m) ~~(1 \leq m \leq n-1).
\end{equation}
\end{lemma}

 The proof of Lemma \ref{lemma} is given in the Appendix. By the identity (\ref{formula1}) in Lemma \ref{lemma} , we can show that the maximum of $H_{n,\beta}(\theta)$ is

\begin{equation} \label{H_n}
H_{n,\beta}(\alpha)= -\frac{\beta}{2}\sum_{i \neq j} \frac{1}{\sin^2 \left( \frac{\pi (i-j)}{n}\right)}=-\frac{(n^3-n) \beta}{6}.
\end{equation}

 Next lemma shows that typically $H_{n, \beta}(\theta)$ is not far from $H_{n, \beta}(\alpha)$.

\begin{lemma}\label{omega}
For any $C>1$, define 
\begin{equation} \label{Theta}
\Theta=\left \{ \theta \in \T~|~  H_{n,\beta}(\alpha)-H_{n,\beta}(\theta) \leq  C n \log n  \right\}.
\end{equation}
 Then there exists some $c>0$, such that
$$\Prob(\Theta) \geq 1- n^{-cn}.$$
\end{lemma}

\begin{remark}
 If we choose $C=1$, then the condition on the set $\Theta$ should be modified as $H_{n,\beta}(\alpha)-H_{n,\beta}(\theta) \leq n \log n -C'n$ for some $C'>0$.
\end{remark}

 Using Lemma \ref{integral} and Lemma \ref{omega}, we have

\begin{lemma}\label{xinomega}
For any $C>1$, define $\Theta$ as in (\ref{Theta}). If $\theta \in \Theta$,  then there exists some positive constant $C_0$ such that
\begin{equation} \label{}
\sum_{i \neq j} \frac{(x_i-x_j)^2}{n^4 \sin^4 \left( \frac{\pi(i-j)}{n} \right)} \leq C_0 n \log^{3} n.
\end{equation}

 Moreover, for all $0 \leq i \neq j \leq n-1$, 
\begin{equation} \label{}
|x_i-x_j| \leq C_0 |i-j|_o n^{\frac{1}{2}} \log^{\frac{3}{2}} n,
\end{equation}
where
\begin{equation}\label{dist}
|i-j|_o=\min \{ |i-j|, n-|i-j|\}.
\end{equation}
\end{lemma}

 Taking the Taylor expansion of the joint probability function $q(\theta)$ around the critical configuration $\theta=\alpha$, we have that for some $\delta \in [0,1]$, 
\begin{equation} \label{taylor}
q(\theta)=\frac{1}{Z_n} e^{H_{n,\beta}(\alpha)} \exp \left\{  \frac{\beta}{4} \sum_{i \neq j} \frac{-\frac{3}{2}+\sin^2 \left(\frac{\pi(i-j)}{n}\right)}{ n^4 \sin^4 \left( \frac{\pi(i-j)}{n}  \right)} (x_i-x_j)^2  \right\}
\end{equation}
$$\times \exp\left\{\frac{\beta}{12} \left [\sum_{i \neq j} \frac{3 \cos \left( \frac{\pi(i-j)}{n} +\delta \frac{x_i-x_j}{2n^2} \right)}{ n^6 \sin^5 \left( \frac{\pi(i-j)}{n} +\delta \frac{x_i-x_j}{2n^2} \right)}-\frac{\cos \left( \frac{\pi(i-j)}{n} +\delta \frac{x_i-x_j}{2n^2} \right)}{ n^6 \sin^3 \left( \frac{\pi(i-j)}{n} +\delta \frac{x_i-x_j}{2n^2} \right)} \right ] (x_i-x_j)^3 \right\}.$$

 Denote the quadratic term by
\begin{equation} \label{G}
G(x) \defeq \frac{\beta}{4} \sum_{i \neq j} \frac{-\frac{3}{2}+\sin^2 \left(\frac{\pi(i-j)}{n}\right)}{ n^4 \sin^4 \left( \frac{\pi(i-j)}{n}  \right)} (x_i-x_j)^2,
\end{equation}
and the cubic term as 
\begin{equation} \label{F}
F(x) \defeq \frac{\beta}{12} \left [\sum_{i \neq j} \frac{3 \cos \left( \frac{\pi(i-j)}{n} +\delta \frac{x_i-x_j}{2n^2} \right)}{ n^6 \sin^5 \left( \frac{\pi(i-j)}{n} +\delta \frac{x_i-x_j}{2n^2} \right)}-\frac{\cos \left( \frac{\pi(i-j)}{n} +\delta \frac{x_i-x_j}{2n^2} \right)}{ n^6 \sin^3 \left( \frac{\pi(i-j)}{n} +\delta \frac{x_i-x_j}{2n^2} \right)} \right ] (x_i-x_j)^3.
\end{equation}

 Using (\ref{theta}), consider the change of variable $\theta \rightarrow (x, \psi)$, where $x$ is a degenerate vector on the hyperplane $\Gamma$,
\begin{equation} \label{Gamma}
\Gamma=\left\{x \in \R^n : \sum_{i=0}^{n-1} x_i=0\right\}.
\end{equation}

 Let 
\begin{equation}\label{f}
f(x)=q(\theta(x, \phi)).
\end{equation}

 Note that the joint probability density $f$ only depends on $x$. But the domain $\Omega$ depends on both $x$ and $\psi$. 
If $ \theta \in \T $, then
$$ (x, \psi) \in \Omega=\left\{ \Gamma \times [-\pi+\frac{\pi}{n},\pi+\frac{\pi}{n}]: x_{i}-x_{i-1} \geq -2 \pi n;  -\frac{x_0}{n^2} \leq \psi \leq \frac{2 \pi}{n}-\frac{x_{n-1}}{n^2} \right\}.$$
Thus, the marginal density function for $x$ is
\begin{equation}\label{px'}
p(x)=\int_{-\frac{x_0}{n^2}}^{ \frac{2 \pi}{n}-\frac{x_{n-1}}{n^2}} f(x) d\psi=\left( \frac{2 \pi}{n}-\frac{x_{n-1}-x_0}{n^2}\right) f(x),
\end{equation}
where 
\begin{equation} \label{Gamma'}
x \in \Lambda= \left\{x \in \Gamma :  x_{i}-x_{i-1} \geq -2 \pi n; x_0 \geq - \pi n(n+1); x_{n-1} \leq \pi n(n+1); x_{n-1}-x_0 \leq 2 \pi n \right\}.
\end{equation}
Otherwise, if $x \in \Lambda^c$, $p(x)=0.$  For any measurable subset $A \subset \Gamma$, denote
\begin{equation}\label{Probx}
\Prob_x(A)=\int_{A} p(x) dx.
\end{equation}

 It follows from Lemma \ref{omega} and Lemma \ref{xinomega} that there exists a subset of $\Lambda$, namely\\
\begin{equation} \label{Gamma_o}
\Gamma_D=\left\{ x \in \Gamma:  \max_{i \neq j} \frac{|x_i-x_{j}|}{|i-j|_o} \leq D n^{\frac{1}{2}} \log^{\frac{3}{2}} n; x_0 \geq - \pi n(n+1); x_{n-1} \leq \pi n(n+1) \right\},
\end{equation}
such that 
\begin{equation} \label{xgamma}
\Prob_x( \Gamma_D^c) \leq n^{-cn}.
\end{equation}

 In addition, if $x \in \Gamma_D$, the probability density of $x$ can be written as 
\begin{equation} \label{px}
p(x) =\frac{1}{{\tilde{Z}}_n}   e^{   G(x) +F(x) }(1+o_n(1)),
\end{equation}
where 
\begin{equation} \label{tZ_n}
{\tilde{Z}_n}=\int_{\Gamma_D} e^{G(x)+F(x)} dx.
\end{equation}

 Let us define a Gaussian distribution on the hyperplane $\Gamma$ by its density
\begin{equation} \label{pg}
p_g(x)=\frac{1}{Z_g} e^{G(x) } ,
\end{equation}
where 
\begin{equation} \label{Z_g}
Z_g=\int_{\Gamma} e^{G(x)} dx.
\end{equation}
For any measurable subset $A \subset \Gamma$, denote 
\begin{equation}\label{Probg}
\Prob_g(A)=\int_{A} p_g(x) dx.
\end{equation}
We use $\E_g$ and $\V_g$ to denote the expectation and variance taken under this Gaussian probability measure.

\begin{remark}
Refining the argument used in the proof of Lemma \ref{xinomega}, we can also prove that there exists a subset $\Gamma_{D'} \subset \Gamma$ such that (1) and (2) hold where
\begin{enumerate}
\item If $x \in \Gamma_{D'}$, we have $\max_{j} |x_j| \ll n$, and thus $\phi \sim \frac{1}{n}$. 
\item $\Prob_x(\Gamma_{D'}^c)=o_n(1)$. 
\end{enumerate}
In addition, let $\tilde{\psi}=n \psi$. Then $$p(\tilde{\psi} |x)=\frac{f(x)}{\int_{-\frac{x_0}{n^2}}^{2 \pi-\frac{x_{n-1}}{n}} f(x) d \tilde{\psi} }=\frac{1}{2 \pi-\frac{x_{n-1}-x_0}{n}}=\frac{1}{2 \pi}(1+o_n(1)).$$
One can show that $n \psi$ and $x$ are asymptotically independent from each other and $n \psi$ converges to the uniform distribution on $[0,2 \pi]$. However, we are not going to use this in the paper.
\end{remark}

\section{Main theorems}
In this section, we formulate our main results. We start with an auxiliary proposition. Recall that we have defined $\Prob_x$ and $\Prob_g$ in (\ref{px})-(\ref{tZ_n}), (\ref{Probx}) and (\ref{pg})-(\ref{Probg}) respectively. Also, $F(x)$ is defined in (\ref{F}).
\begin{proposition}\label{tech1}
There exists a subset ${\Gamma'_{}} \subset \Gamma$,  such that 
$$\Prob_x({\Gamma'_{}})=1-o_n(1), ~~~ \Prob_g({\Gamma'_{}})=1-o_n(1),$$ and
$$\sup_{x \in \Gamma'} F(x)=o_n(1).$$
\end{proposition}

 Proposition \ref{tech1} immediately implies that the total variation distance between $\Prob_x$ and $\Prob_g$ goes to zero as $n$ goes to infinity.
\begin{theorem} \label{tech2}
$$\sup_{A \subset \Gamma} |\Prob_x(A)-\Prob_g(A)|=o_n(1),$$
where the supremum at the LHS is taken over all measurable subsets $A \subset \Gamma$.
\end{theorem}

 Based on Theorem \ref{tech2}, we obtain the main theorem. For each fixed $n$, we construct a random function in $ C[0, 2 \pi]$, denoted as  $\zeta_n(t)$, by letting $\zeta_n \left(  \frac{2 \pi j}{n}\right) =\frac{x_j}{\sqrt{n}}$ and then connecting these lattice points with straight segments. Define the limiting random function to be
\begin{equation} \label{zeta}
\zeta(t)=\sqrt{\frac{2}{\beta}}\Real \left(  \sum_{k=1}^{\infty} \frac{1}{k} e^{ikt} Z_k  \right)=\sqrt{\frac{2}{\beta}}\sum_{k=1}^{\infty} \left( \frac{\cos kt}{k} X_k- \frac{\sin kt}{k} Y_k \right),
\end{equation}
where $X_k,Y_k$ are i.i.d. real standard Gaussian random variables, and $Z_k$ are i.i.d. complex standard Gaussian random variables. Note that this is a well defined random function since the variance is bounded. It can be viewed as an analogue of $(\ref{T})$ in the CUE case. We have:

\begin{theorem}\label{weak}
$\zeta_n(t)$ converges to $\zeta(t)$ in finite dimensional distributions. Furthermore, the functional convergence takes place. In other words, $\zeta_n(t)$ converges to $\zeta(t)$ in distribution weakly on the space $C[0, 2\pi]$.
\end{theorem}

 Finally, we finish this section by formulating two corollaries.
\begin{corollary}\label{CLT}
Consider periodic function $g$ on $S^1$ with complex Fourier coefficients $\{c_k\}_{k \geq 0}$,  where $c_k= \frac{1}{2 \pi}\int_{0}^{2\pi} g(x) e^{ikx} dx $, such that $\sum_{k=-\infty}^{\infty} |k|^{\frac{3}{2}} |c_k| < \infty$,
 then 
\begin{equation} \label{}
\E \exp \left( { i t \left( \sqrt{n} \sum_{j=0}^{n-1} g(\theta_j) -n^{\frac{3}{2}} c_0  \right)} \right)=\exp \left( { -\frac{t^2}{ \beta} \sum_{k=1}^{\infty} |c_k|^2} \right) (1+o_n(1)).
\end{equation}
In other words, 
$\sqrt{n} \sum_{j=0}^{n-1} g(\theta_j)-n^{\frac{3}{2}} c_0$ converges in distribution to $N \left(0, \frac{2}{ \beta} \sum_{k=1}^{\infty} |c_k|^2 \right)$.
\end{corollary}

\begin{remark}
It follows from the proof of Corollary \ref{CLT} that the statement holds under a milder condition on the the decay of Fourier series of $f$, i.e. $\sum_{k=1}^{\infty} |c_{kn}| \ll n^{-\frac{3}{2}}$. Furthermore, Corollary \ref{CLT} is expected to hold when $f \in C^{1+\epsilon}(\abc)$.
\end{remark}

\begin{corollary}\label{max}
$\max_{0 \leq i \leq n-1} \left|\frac{x_j}{\sqrt{n}}\right|$ converges in distribution to $\sup_{t \in [0, 2\pi]} |\zeta(t)|$.
\end{corollary}

 To simplify the notations, the proofs of these results are written for $\beta=2$. The general case $\beta>0$ is essentially identical.

\section{Proofs of the Lemmas in Section 2}

 We start this section by proving Lemma \ref{integral}.

\begin{proof}(Lemma \ref{integral})
 Define $$\phi(\tau) \defeq H_{n,2}(\alpha+\tau \frac{x}{n^2}).$$
We compute its first and second derivative with respect to $\tau$,
\begin{equation} \label{first}
\phi'(\tau)=\sum_{i \neq j} \frac{\cos \left( \frac{\pi(i-j)}{n} +\tau \frac{x_i-x_j}{2n^2} \right)}{ \sin^3 \left( \frac{\pi(i-j)}{n} +\tau \frac{x_i-x_j}{2n^2} \right)} \cdot \frac{x_i-x_j}{n^2};
\end{equation}

\begin{equation} \label{second}
\phi''(\tau)=- \sum_{i \neq j} \frac{\frac{1}{2}+\cos^2 \left( \frac{\pi(i-j)}{n} +\tau \frac{x_i-x_j}{2n^2} \right)}{ \sin^4 \left( \frac{\pi(i-j)}{n} +\tau \frac{x_i-x_j}{2n^2} \right)} \cdot \left( \frac{x_i-x_j}{n^2} \right)^2.
\end{equation}

 Note that $$\phi'(0)=\sum_{i \neq j} \frac{\cos \left( \frac{\pi(i-j)}{n}  \right)}{ \sin^3 \left( \frac{\pi(i-j)}{n}  \right)} \cdot \frac{x_i-x_j}{n^2}=\frac{1}{n^2} \left( \sum_{i \neq j} \frac{\cos \left( \frac{\pi(i-j)}{n}  \right)}{ \sin^3 \left( \frac{\pi(i-j)}{n}  \right)} \cdot x_i-\sum_{i \neq j} \frac{\cos \left( \frac{\pi(i-j)}{n}  \right)}{ \sin^3 \left( \frac{\pi(i-j)}{n}  \right)} \cdot x_j \right)$$
$$=\frac{1}{n^2}  \sum_{i}\left( \sum_{k=1}^{n-1} \frac{\cos \frac{\pi k}{n} }{ \sin^3  \frac{\pi k}{n}  } \right) x_i- \frac{1}{n^2}  \sum_{j} \left( \sum_{l=1}^{n-1} \frac{\cos  \frac{\pi l}{n}  }{ \sin^3  \frac{\pi l}{n} } \right) x_j =0,$$ and $\phi''(0) \leq 0$. Thus,
$$H_{n,2}(\alpha)-H_{n,2}(\theta)=\phi(0)-\phi(1)=-\int_0^1 (1-\tau) \phi''(\tau) d \tau$$
$$=\sum_{i \neq j}\frac{(x_i-x_j)^2}{n^4} \int_0^1 \frac{\frac{1}{2}+\cos^2 \left( \frac{\pi(i-j)}{n} +\tau \frac{x_i-x_j}{2n^2} \right)}{ \sin^4 \left( \frac{\pi(i-j)}{n} +\tau \frac{x_i-x_j}{2n^2} \right)} \cdot (1- \tau) d \tau \geq 0.$$
This implies that $H_{n,2}(\theta)$ obtains its maximum at $\theta=\alpha$.

\end{proof}

 Next, we turn our attention to proving Lemma \ref{omega}.

\begin{proof}(Lemma \ref{omega})
It follows from the definition of $\Theta$ in (\ref{Theta}) and the trigonometric identity (\ref{H_n}) that
\begin{equation} \label{z3}
\Prob(\Theta^{c})=\frac{1}{Z_n} \int_{\Omega^c} e^{H_{n,2}(\theta)} d \theta  \leq \left( e^{\frac{n^3-n}{3}-C n \log n} \right) \frac{1}{Z_n} \mu(\T) =\frac{(2 \pi)^n}{n!} e^{\frac{n^3-n}{3}-C n \log n}  \frac{1}{Z_n},
\end{equation}
where $\mu$ denote the Lebesgue measure on $\R^n$.
Choose any $0< C' <C$ and define a subset of $\Theta$,
\begin{equation} \label{}
\Theta'=\left \{ \theta \in \T~|~  H_{n,2}(\alpha)-H_{n,2}(\theta) \leq C' n \log n   \right\}.
\end{equation}
Then 
\begin{equation} \label{z1}
Z_n=\int_{\T} e^{H_{n,2}(\vec{\theta})} d \theta \geq \int_{\Theta'} e^{H_{n,2}(\alpha)-C' n \log n} d \theta=e^{\frac{n^3-n}{3}-C' n \log n} \mu(\Theta').
\end{equation}
Note that if  there exists some constant $M>0$ such that $|x_j| \leq M$  , then 
$$H_{n,2}(\alpha)-H_{n,2}(\theta) \leq \sum_{i \neq j}\frac{4 M^2}{ n^4 \sin^4 \left( \frac{\pi(i-j)}{n} \right)}  \leq 4 M^2 \sum_{i \neq j} \frac{1}{ \pi^4 |i-j|^4} \leq C'n \log n.$$
Therefore, $$\left \{  \theta \in \T~|~ \theta_i-\alpha_i \leq \frac{M}{n^2} ,~~\forall~0 \leq i \leq n-1 \right\} \subset \Theta',$$
and thus the Lebesgue measure of the set $\Theta'$ can be bounded from below as
\begin{equation} \label{z2}
\mu(\Theta') \geq \left(\frac{M}{n^2}\right)^n = e^{n \log M  -2 n \log n}.
\end{equation}
Therefore, by (\ref{z1}) and (\ref{z2}),
\begin{equation}
{Z_n} \geq e^{\frac{n^3-n}{3}-C' n \log n+n \log M  -2 n \log n}.
\end{equation}
Combining it with (\ref{H_n}) and (\ref{z3}), we obtain
\begin{equation}
\Prob(\Theta^{c}) \leq \frac{(2 \pi)^n}{n!} e^{-(C-C') n \log n-n \log M +2 n \log n}  \leq C'' e^{-(C-C'-1) n \log n}=o_n(1),
\end{equation}
provided $C>1$ and $0<C'<C-1$.
\end{proof}

 Using Lemma \ref{omega}, we finish this section by giving the proof of Lemma \ref{xinomega}.

\begin{proof}(Lemma \ref{xinomega})

 By Lemma \ref{omega}, if $\theta \in \Theta$, then for some constant $C>1$,

\begin{equation} \label{ine}
\sum_{i \neq j}\frac{(x_i-x_j)^2}{n^4} \int_0^1 \frac{(1-\tau) d \tau}{ \sin^4 \left( \frac{\pi(i-j)}{n} +\tau \frac{x_i-x_j}{2n^2} \right)} \leq C n \log n. 
\end{equation}

 Let $I=\left\{ (i,j)~|~ 0 \leq i \neq j \leq n-1  \right\}$, $I_1=\left\{ (i,j) \in I~|~ |x_i-x_j| < n \eta_n |i-j|_o \right\}$, and $I_2=I \setminus  I_1$. Let $\eta_n \gg 1$. Then by (\ref{ine}) we have 
$$ C n \log n \geq \sum_{I_1}\frac{(x_i-x_j)^2}{n^4} \int_0^1 \frac{(1-\tau) d \tau}{ \sin^4 \left( \frac{\pi(i-j)}{n} +\tau \frac{x_i-x_j}{2n^2} \right)}  \geq  \sum_{I_1}\frac{(x_i-x_j)^2}{n^4} \int_0^1 \frac{(1-\tau) d \tau}{\left(\frac{ \pi |i-j|_o}{n} +\tau \frac{|x_i-x_j|}{2n^2} \right)^4} $$
$$\geq C' \sum_{I_1} (x_i-x_j)^2 \frac{1}{\eta_n^4 |i-j|_o^4}.$$
Thus, \begin{equation} \label{}
\sum_{I_1}  \frac{(x_i-x_j)^2}{ |i-j|_o^4} \leq C'' n \eta_n^4 \log n.
\end{equation}
Next, it can be shown that $I_2 =\emptyset$ for $\eta_n$ satisfying 
\begin{equation}\label{etanew}
\eta_n \geq M \log^{\frac{1}{2}} n
\end{equation}
for sufficient large $M>0$.  Note that
\begin{equation} \label{int}
 \int_0^1 \frac{(1-\tau) d \tau}{ \sin^4 \left( \frac{\pi (i-j)}{n} +\tau \frac{x_i-x_j}{2n^2} \right)}  \geq \int_0^1 \frac{(1-\tau) d \tau}{  \left( \frac{\pi |i-j|_o}{n} +\tau \frac{|x_i-x_j|}{2n^2} \right)^4} =\frac{\frac{3 \pi |i-j|_o}{n}+\frac{|x_i-x_j|}{n^2}}{6 \left( \frac{\pi |i-j|_o}{n} \right)^3 \left( \frac{\pi |i-j|_o}{n}+\frac{|x_i-x_j|}{2n^2}\right)^2}.
\end{equation}
If $(i,j) \in I_2$, then $$\frac{|x_i-x_j|}{n^2} \geq \eta_n \frac{|i-j|_o}{n},$$
and
\begin{equation} \label{l}
\mbox{RHS of } (\ref{int}) \geq const \frac{ \eta_n \frac{|i-j|_o}{n}}{\left( \frac{|i-j|_o}{n} \right)^3 \left( \frac{|x_i-x_j|}{n^2} \right)^2 } \geq const \frac{ \eta_n n^6}{\left( |i-j|_o \right)^2 \left(|x_i-x_j|\right)^2 }.
\end{equation}
Note that by the triangle inequality,  if $(i,j) \in I_2$, then there exists at least $|i-j|_o$ index pairs belonging to $I_2$, in the form of $(i, k)$ or $(k, j)$ where k is between i and j.
Thus by (\ref{ine}), (\ref{int}) and (\ref{l}), $$C n \log n \geq C' \sum_{I_2} \frac{ \eta_n n^2}{ |i-j|_o^2 } \geq C' n^2 \eta_n \frac{1}{|i-j|_o}.$$
This implies that
$ |i-j|_o \geq C'' n \eta_n \log^{-1} n$, and thus $ |x_i-x_j| \geq C'' n^2 \eta^2_n \log^{-1} n.$
Due to (\ref{etanew}), we obtain
$$|x_i-x_j| \geq C'' M^2 n^2 .$$ With a sufficient large $M$, the last inequality contradicts that $|x_i-x_j| \leq 2 \pi n^2$.
Therefore $I_2 =\emptyset$ and there exists some positive constant $C_0$ such that
\begin{equation} \label{}
\sum_{I}  \frac{(x_i-x_j)^2}{ |i-j|_o^4} \leq C_0 n \log^3 n.
\end{equation}

 Furthermore, denote $x_{n+i}=x_i,~i \geq 0$, then for $0 \leq j \leq n-1$, $$(x_{j+1}-x_j)^2 \leq C_0 n \log^{3} n.$$
Therefore, by the triangle inequality, we have 
$$ |x_i-x_j| \leq C_0 |i-j|_o n^{\frac{1}{2}} \log^{\frac{3}{2}} n.$$

 Finally, since $$\frac{2}{\pi} \leq \frac{\sin x}{x} \leq 1, ~~0 < x \leq \frac{\pi}{2},$$
we have 
\begin{equation} \label{}
\sum_{i \neq j} \frac{(x_i-x_j)^2}{n^4 \sin^4 \left( \frac{\pi(i-j)}{n} \right)} \leq C_0' n \log^{3} n.
\end{equation}

\end{proof}

\section{Estimate of the multivariate Gaussian distribution}
 Note that $G(x)$ defined in (\ref{G}) can be written in the quadratic form of $ -\frac{1}{2} x^{T} A x,$
where 
\begin{equation} \label{}
A_{i,j}=-\frac{3}{n^4 \sin^4 \left( \frac{\pi(i-j)}{n} \right)}+\frac{2}{n^4 \sin^2 \left( \frac{\pi(i-j)}{n} \right)}~~~(i \neq j),
\end{equation}
and, by the identities (\ref{formula1}) and (\ref{formula2}) in Lemma \ref{lemma}, 
\begin{equation} \label{}
A_{ii}=\frac{(n^2-1)(n^2+11)}{15n^4}-\frac{2n^2-2}{3n^4}=-\sum_{j \neq  i} A_{ij}.
\end{equation}
Since $A$ is not invertible, $p_g$ defined in (\ref{pg})-(\ref{Z_g}) can be viewed as a degenerate Gaussian distribution on the hyperplane $\Gamma$ defined in (\ref{Gamma}), $$p_g(x) =\frac{1}{Z_g} e^{-\frac{1}{2} x^{T} A x},$$ where $Z_g=\int_{\Gamma} e^{-\frac{1}{2} x^{T} A x} d x$.

 Next, we aim to explore the covariance structure of this Gaussian distribution. Note that $A$ is a circular matrix generated by the vector $(A_{0,0}, A_{0,1}, \cdots, A_{0,n-1})$. Therefore its normalized eigenvectors can be chosen as
\begin{equation} \label{eigenvector}
{v}_k=\frac{1}{\sqrt{n}} \left( \omega_k^0, \omega_k^1, \cdots, \omega_k^{n-1}   \right)~~(k=0,1, \cdots, n-1),
\end{equation}
where $\omega_k=e^{\frac{2 \pi i k}{n}}$. By using the identities (\ref{formula3}) and (\ref{formula4}) in Lemma \ref{lemma}, the corresponding eigenvalues are given by
$$\lambda_{k}=A_{0,0}+A_{0,1} \omega_k+A_{0,2} \omega^2_k+\cdots+A_{0,n-1} \omega_{k}^{n-1}$$
$$=\frac{(n^2-1)(n^2+11)}{15n^4}-\frac{2n^2-2}{3n^4}-\frac{3}{n^4} \sum_{j=1}^{n-1} \frac{\cos \frac{2 \pi j k}{n}}{\sin^4 \frac{j \pi}{n}}+\frac{2}{n^4} \sum_{j=1}^{n-1} \frac{\cos \frac{2 \pi j k}{n}}{\sin^2 \frac{j \pi}{n}}$$
\begin{equation} \label{eigenvalue}=\frac{6}{n^4} \sum_{j=1}^{n-1} \frac{\sin^2 \frac{ \pi j k}{n}}{\sin^4 \frac{j \pi}{n}}-\frac{4}{n^4} \sum_{j=1}^{n-1} \frac{\sin^2 \frac{ \pi j k}{n}}{\sin^2 \frac{j \pi}{n}}= \frac{2 k^2(n-k)^2}{n^4}, ~~0 \leq k \leq n-1.
\end{equation}

 Define 
\begin{equation} \label{u}
U=({v}_0^T, {v}^T_1, \cdots, {v}^T_{n-1}),
\end{equation}
 then we have $U^*U=1$ and $ A=U \Lambda U^*$, where $\Lambda$ is the diagonal matrix generated by $(\lambda_0,\lambda_1, \cdots, \lambda_{n-1})$. Let ${s}= U^* {x}$.  Then for $1 \leq k,j \leq n-1$,  
\begin{equation} \label{covofs}
\E_g s_k \overline{s_k}=\frac{1}{\lambda_k}= \frac{n^4}{2 k^2(n-k)^2}~~ \mbox{and}~~~ \E s_k \overline{s_j}=0~~(k \neq j).
\end{equation}

 Note that $s_0=0$ because of the definition of $\Gamma$ in (\ref{Gamma}). We also note that  
\begin{equation} \label{s1}
s_{j}=\overline{s_{n-j}} ~~\mbox{if}~~j >\frac{n-1}{2}.
\end{equation}
For simplicity, we assume that $n$ is odd. (the even case can be treated in a similar way). Then $(s_1, \cdots, s_{\frac{n-1}{2}})$ are $\frac{n-1}{2}$ independent complex Gaussian random variables. In particular, we can write 
\begin{equation} \label{s2}
s_k=\frac{n^2}{2 k(n-k)} X_k+i \frac{n^2}{2 k(n-k)} Y_k,~~~\mbox{for}~~1 \leq k \leq \frac{n-1}{2},
\end{equation}
 where $\{X_k\}$ and $\{Y_k\}$ are independent real standard Gaussian variables.\\

 Since $ x=U {s}$, we can compute the covariance structure for $x$,
\begin{equation} \label{covofx}
\E_g x_k \overline{x_j}=\frac{1}{n} \sum_{m=1}^{n-1}\frac{1}{\lambda_m} e^{\frac{2 \pi i}{n} m(k-j)}=\frac{1}{2n} \sum_{m=1}^{n-1} e^{\frac{2 \pi i}{n} m(k-j)} \frac{n^4}{m^2 (n-m)^2}.
\end{equation}

 In particular, 
 \begin{equation}\label{varofx}
 \V_g(x_k) \sim n.
 \end{equation}
 In addition,
$$x_j=\frac{1}{\sqrt{n}} \sum_{k=0}^{n-1} e^{\frac{2 \pi i j k}{n}} s_k=\frac{1}{\sqrt{n}}  \sum_{k=1}^{\frac{n-1}{2}} e^{\frac{2 \pi i j k}{n}} s_k+\frac{1}{\sqrt{n}}  \sum_{k=1}^{\frac{n-1}{2}} e^{-\frac{2 \pi i j k}{n}} \bar{s}_{k}=\frac{2}{\sqrt{n}}\sum_{k=1}^{\frac{n-1}{2}} \mathfrak{Re} \left( e^{\frac{2 \pi i j k}{n}} s_{k} \right)$$
\begin{equation}\label{C}
=\frac{2}{\sqrt{n}} \sum_{k=1}^{\frac{n-1}{2}} \left( \cos \left({\frac{2 \pi  j k}{n}}\right) \frac{n^2}{2 k(n-k)} X_k-\sin \left( {\frac{2 \pi  j k}{n}}\right) \frac{n^2}{2 k(n-k)} Y_k \right).
\end{equation}

For $0 \leq j \leq n-1$, $0 \leq l \leq n-1$, let 
\begin{equation} \label{xi}
\xi^{(l)}_j=x_{j+l}-x_j, ~\mbox{where}~~x_{n+i}=x_i,~i \geq 0.
\end{equation}
It is also useful for us to compute the covariance of $\xi^{(l)}_{j}$ and $\xi^{(l)}_{k}$.

\begin{proposition}\label{covofxi}
There exists some positive constant $C$ which is independent of other parameters such that
\begin{equation} \label{covofxi3}
|\E_g \xi^{(l)}_k {\bar{\xi}^{(l)}_j}| \leq  C \min \left\{ l,  \frac{ l^2}{|k-j|_o} \right\}.
\end{equation}
In particular, 
\begin{equation} \label{covofxi4}
\V_g (\xi^{(l)}_k ) \leq C l.
\end{equation}
\end{proposition}

\begin{proof}
 By (\ref{covofx}), we have
$$\E_g \xi^{(l)}_k {\bar{\xi}^{(l)}_j}=\E_g (x_{k+l}-x_k)(\bar{x}_{j+l}-\bar{x}_j)=\E_g x_{k+l} \bar{x}_{j+l}+\E_g x_{k} \bar{x}_{j}-\E_g x_{k+l}\bar{x}_{j}-\E_g x_{k} \bar{x}_{j+l}$$
\begin{equation} \label{covofxi1}
=\frac{1}{n} \sum_{m=1}^{n-1} \frac{1}{\lambda_m} e^{\frac{2 \pi i}{n} m(k-j)} \left( 2-e^{\frac{2 \pi i ml}{n}}-e^{-\frac{2 \pi i ml}{n}} \right)=\frac{2}{n} \sum_{m=1}^{n-1} \frac{\sin^2 \left( \frac{\pi m l}{n} \right)}{\left(\frac{m}{n}\right)^2 \left(1-\left(\frac{m}{n}\right) \right)^2} e^{ 2 \pi i (k-j)\frac{m}{n}} .
\end{equation}
For the RHS of (\ref{covofxi1}), we can find an upper bound as
\begin{equation} \label{covofxi2}
|\E_g \xi^{(l)}_k {\bar{\xi}^{(l)}_j}|\leq \frac{2}{n} \sum_{m=1}^{n-1} \frac{\sin^2 \left( \frac{\pi m l}{n} \right)}{\left(\frac{m}{n}\right)^2 \left(1-\left(\frac{m}{n}\right) \right)^2}=\frac{4}{n} \sum_{m=1}^{\frac{n-1}{2}} \frac{\sin^2 \left( \frac{\pi m l}{n} \right)}{\left(\frac{m}{n}\right)^2 \left(1-\left(\frac{m}{n}\right) \right)^2}  \leq \frac{16l}{3}  \sum_{m=1}^{\frac{n-1}{2}} \frac{\sin^2 \left( \frac{\pi m l}{n} \right)}{\left(\frac{lm}{n}\right)^2 } \cdot \frac{l}{n}.
\end{equation}

The RHS of (\ref{covofxi2}) can be viewed as a Riemann sum of the function $\frac{\sin^2 {\pi x}}{x^2}$ corresponding to the evenly-spaced partition over $[0, \frac{l}{2}]$ with the subintervals of length $\frac{l}{n-1}$. Since the function $\frac{\sin^2 {\pi x}}{x^2}$ can be bounded from above by a monotone function $m(x)$ defined as

\begin{equation*}
    m(x)=
   \begin{cases}
   \pi^2 &\mbox{if $0 \leq x \leq 1$ }\\
   \frac{1}{x^2} &\mbox{if $x>1$},
   \end{cases}
  \end{equation*} 
the RHS of (\ref{covofxi2}) can be bounded by a Riemann sum of $m(x)$. Note that $m(x)$ is monotone,  the error between its upper and lower Riemann sum is at most of the order $\frac{l}{n}$. Then there exists a universal constant $C>0$, such that $$|\E_g \xi^{(l)}_k {\bar{\xi}^{(l)}_j}| \leq \frac{16}{3} l \int_0^{\frac{l}{2}} m(x) dx +O\left(\frac{l}{n}\right) \leq \frac{16l}{3}  \int_0^{\infty} m(x) dx (1+o(1)) \leq Cl.$$

 For large $|k-j|_o$, the heavy oscillation of the exponential term leads to cancellations between terms in the expression (\ref{covofxi1}) for $\E_g \xi^{(l)}_k {\bar{\xi}^{(l)}_j}$. Thus the upper bound that we have obtained above is not sharp in this case. Let $a_m=\frac{\sin^2 \left( \frac{\pi m l}{n} \right)}{\left(\frac{m}{n}\right)^2 \left(1-\left(\frac{m}{n}\right) \right)^2} $, and $b_m=e^{ 2 \pi i (k-j)\frac{m}{n}} $. By summation by parts, we have
$$\sum_{m=1}^{n-1} a_m b_m=\sum_{m=1}^{n-2} (a_m-a_{m+1}) \left( \sum_{p=1}^{m} b_p\right)+a_{n-1} \sum_{p=1}^{n-1} b_p.$$
Then

$$|\E_g \xi^{(l)}_k {\bar{\xi}^{(l)}_j}| \leq \frac{2}{n} \sum_{m=1}^{n-2} |a_m-a_{m+1}| \left|\frac{1-e^{\frac{2 \pi i (k-j) m}{n}}}{1-e^{\frac{2 \pi i (k-j) }{n}}}\right| +O\left( \frac{l^2}{n} \right).$$
By differentiating the function $f(x)=\frac{\sin^2 \pi l x }{x^2 (1-x)^2}$, we find that the derivative is at most of the order $l^2$ and we have $$|a_m-a_{m+1}| =\left|f\left(\frac{m}{n}\right)-f\left(\frac{m+1}{n}\right)\right| \leq  \frac{C' l^2}{n} .$$
Using the inequality $$\left|\frac{1-e^{\frac{2 \pi i (k-j) m}{n}}}{1-e^{\frac{2 \pi i (k-j) }{n}}}\right| \leq \left|\frac{2}{1-e^{\frac{2 \pi i (k-j) }{n}}}\right| \leq \frac{C' n}{|k-j|_o},$$
we have $$|\E_g \xi^{(l)}_k {\bar{\xi}^{(l)}_j}| \leq \frac{C' l^2}{|k-j|_o}.$$
 This finishes the proof of the Proposition \ref{covofxi}.

\end{proof}

 Combining (\ref{covofxi4}) and (\ref{varofx}) with (\ref{Gamma_o}), we have
\begin{lemma}\label{ggamma}
There exist some positive constants $c_1$, $c_2$, such that
 \begin{equation} 
\Prob_g( \Gamma^c_D) \leq 2 e^{-c_1 n^3}+n^2 e^{-c_2 n \log^3 n}.
\end{equation}
\end{lemma}

\section{Proof of Proposition \ref{tech1} and Theorem \ref{tech2}} 
In this section, we prove Proposition \ref{tech1} and Theorem \ref{tech2}. We start with some preliminary details.
  If $x=(x_0, \cdots, x_{n-1}) \in \Gamma_D$, then by Lemma \ref{xinomega}, 
$$G(x)=\frac{1}{2} \sum_{i \neq j} \frac{-\frac{3}{2}+\sin^2 \left(\frac{\pi(i-j)}{n}\right)}{ n^4 \sin^4 \left( \frac{\pi(i-j)}{n}  \right)} (x_i-x_j)^2  \sim -\sum_{i>j} \frac{(x_i-x_j)^2}{ |i-j|_o^4} = O \left( -n \log^3 n \right),$$
 and 
\begin{equation} \label{B}
\frac{|x_i-x_j|}{|i-j|_o} =O\left(   n^{\frac{1}{2}} \log^{\frac{3}{2}} n \right).
\end{equation}
Then  $\frac{ (x_i-x_j)}{2n^2} =O \left( \frac{\log^{\frac{3}{2}} n}{ n^{\frac{3}{2}}} \right)$ is negligible, and thus
$$F(x) = \frac{1}{6} \left [\sum_{i \neq j} \frac{3 \cos \left( \frac{\pi(i-j)}{n} +\delta \frac{x_i-x_j}{2n^2} \right)}{ n^6 \sin^5 \left( \frac{\pi(i-j)}{n} +\delta \frac{x_i-x_j}{2n^2} \right)}-\frac{\cos \left( \frac{\pi(i-j)}{n} +\delta \frac{x_i-x_j}{2n^2} \right)}{ n^6 \sin^3 \left( \frac{\pi(i-j)}{n} +\delta \frac{x_i-x_j}{2n^2} \right)} \right ] (x_i-x_j)^3$$
\begin{equation} \label{F1}
\sim \sum_{i \neq j} \left [\frac{(x_i-x_j)^3}{2 n^6 \sin^5 \left( \frac{\pi(i-j)}{n}  \right)}+\frac{(x_i-x_j)^3 }{ 6 n^6 \sin^3 \left( \frac{\pi(i-j)}{n}\right)} \right] \sim  \sum_{i>j} \frac{(x_i-x_j)^3}{n|i-j|_o^5}.
\end{equation}
Comparing $F(x)$ with $G(x)$ and by (\ref{B}), we have 
\begin{equation} \label{FandG}
|F(x)|=|G(x)| O\left(  \max_{i \neq j} \frac{|x_i-x_j|}{n |i-j|_o} \right)=|G(x)| O\left( \frac{\log^{\frac{3}{2}} n}{n^{\frac{1}{2}}} \right) =O \left( n^{\frac{1}{2}} \log^{\frac{9}{2}} n \right).
\end{equation}
In this section, we show that $F(x) = o_n(1)$  with probability $1-o_n(1)$. To be more specific, by (\ref{F1}), we show that 
\begin{equation}\label{goal}
F(x)  \sim \sum_{l=1}^{n-1} \frac{1}{l^5} \left( \frac{1}{n} \sum_{j=0}^{n-l-1} {\xi_j^{(l)}}^3 \right) =o_n(1),
\end{equation}
where $\xi_j^{(l)}$ is defined in (\ref{xi}).\\

We divide the proof of (\ref{goal}) into three parts. The first step is to show that the normalized constant $\tilde{Z}_n$ defined in (\ref{tZ_n}) is not far from $Z_g$ defined in (\ref{Z_g}). This implies that the probability distribution of $x$ is not far from the Gaussian distribution $p_g(x)$ defined in (\ref{pg})-(\ref{Z_g}). The second step is to show that under the Gaussian distribution, $F(x)=o_n(1)$ with probability $1-O(n^{-c})$. The last step is to combine the first two steps and obtain that $F(x) =o_n(1)$ with probability $1-O(n^{-c})$ under the distribution defined in (\ref{px})-(\ref{tZ_n}).

\subsection{Step 1: Comparing $Z_g$ and $\tilde{Z_n}$}
 For reader's convenience, recall the definition of the normalized constants $\tilde{Z}_n$ and $Z_g$.
$${\tilde{Z}_n}=\int_{\Gamma_D} e^{G(x)+F(x)} dx;~~~Z_g=\int_{\Gamma} e^{G(x)} dx.$$
 We start with a lemma. Rescale the Gaussian distribution defined in (\ref{pg}) and define two new Gaussian distributions,
\begin{equation} \label{}
p_{g^{+}} (x) =\frac{1}{Z_g^{+}} e^{G(x)(1+\frac{cB(n)}{n})},~~~~~p_{g^{-}} (x) =\frac{1}{Z_g^{-}} e^{G(x)(1-\frac{cB(n)}{n})}.
\end{equation}

\begin{lemma}\label{scale}
If $B(n) \ll n$, the normalized constants $Z_g^{\pm}$ satisfy
$$Z^{\pm}_g = Z_g e^{\mp\frac{c}{2} B(n)(1+o_n(1))}.$$

\end{lemma}
\begin{proof}

 Change the variable to $\tilde{x}= \sqrt{1-\frac{cB(n)}{n}} x$,  then
 $$Z^{-}_g= \int_{\Gamma} e^{-\frac{1}{2} x^T A x \left( 1- \frac{c B(n)}{n} \right)} dx=\int_{\Gamma} e^{-\frac{1}{2} \tilde{x}^T A \tilde{x}} \prod_{j=1}^{n}  \frac{1}{\sqrt{1-\frac{c B(n)}{n}}} d \tilde{x}=Z_g \prod_{j=1}^{n}  \frac{1}{\sqrt{1-\frac{c B(n)}{n}}}.$$
Note that $$\left ( 1- \frac{c B(n)}{n} \right)^{-\frac{n}{2}} = e^{\frac{c B(n) }{2}(1+o_n(1))} .$$
Similarly, we have $$\left ( 1+ \frac{c B(n)}{n} \right)^{-\frac{n}{2}} = e^{\frac{-c B(n) }{2}(1+o_n(1))} .$$
Thus $$Z^{\pm}_g = Z_g e^{\mp\frac{c}{2} B(n)(1+o_n(1))}.$$

\end{proof}

 Let $B(n)=Dn^{\frac{1}{2}} \log^{\frac{3}{2}} n$. By the definition of $\Gamma_D$ in (\ref{Gamma_o}), we have $ \max_{i \neq j} \frac{|x_i-x_j|}{|i-j|_o} \leq  B(n)$. Then by the first equation of (\ref{FandG}), there exists a universal constant $c>0$, such that 
$$\frac{c B(n)}{n} G(x) \leq F(x) \leq - \frac{c B(n)}{n} G(x).$$
Similar to Lemma \ref{scale}, one can show that 
$$ Z_g e^{-c B(n)} \leq \tilde{Z}_n \leq Z_g e^{c B(n)}.$$
When we proceed to the Step 2 and Step 3 presented below, this estimate will not be sufficient for us to show $F(x) =o_n(1)$ because the exponential term $e^{c B(n)}$ grows faster than any polynomial. In order to get a better upper bound of  $$\max_{i \neq j} \frac{|x_i-x_j|}{|i-j|_o},$$ we use iteration. We need the following two lemmas.

\begin{lemma}\label{step1}
For some sufficient large $ M>0$, let $  M \log^{\frac{1}{2}} n \leq  B \leq D n^{\frac{1}{2}} \log^{\frac{3}{2}} n$. If there exists some $\gamma>0$ such that
\begin{equation} \label{}
\Prob_x \left(x \in \Gamma_D : \max_{i \neq j} \frac{|x_i-x_j|}{|i-j|_o} \geq  B \right) \leq n^{-\gamma}, 
\end{equation}
 then there exists a universal constant $c>0$, we have
\begin{equation} \label{}
Z_g e^{-c B} \leq \tilde{Z}_n \leq Z_g e^{c B}.
\end{equation}
\end{lemma}

\begin{proof}
Denote $A=\{ x \in \Gamma_D : \max_{i \neq j} \frac{|x_i-x_j|}{|i-j|_o} \leq  B\}$ and then $\Prob_x(\Gamma_D \setminus A) \leq n^{-\gamma}$. Because of (\ref{FandG}), if $x \in A$, there exists   a universal constant $c>0$, 
\begin{equation}\label{D}
\frac{c B}{n} G(x) \leq F(x) \leq -\frac{c B}{n} G(x).
\end{equation}
By (\ref{D}), we have
$$\tilde{Z}_n=\int_{\Gamma_D} e^{G(x)+F(x)} dx=\int_{A} e^{G(x)+F(x)} dx+\Prob_x \left( \Gamma_D \setminus A \right) \tilde{Z}_n \leq \int_{A} e^{G(x)(1-\frac{cB}{n})} dx+n^{-\gamma} \tilde{Z}_n.$$
Then
$$(1-n^{-\gamma}) \tilde{Z}_n \leq \int_{\Gamma} e^{G(x)(1-\frac{cB}{n})} dx \leq e^{\frac{cB}{2} (1+o_n(1))} Z_g.$$
Thus 
\begin{equation} \label{lower}
\tilde{Z}_n \leq e^{\frac{cB}{2}(1+o_n(1))} Z_g(1+O(n^{-\gamma})) \leq e^{cB} Z_g.
\end{equation}
For the lower bound, similarly, 
$$\tilde{Z}_n \geq \int_{A} e^{G(x)+F(x)}dx \geq \int_{A} e^{G(x)(1+\frac{c B}{n})} dx=\int_{\Gamma} e^{G(x)(1+ \frac{cB}{n})} dx-\int_{\Gamma \setminus A} e^{G(x)(1+ \frac{c B}{n})} dx $$
$$\geq \int_{\Gamma} e^{G(x)(1+ \frac{cB}{n})} dx- \int_{\Gamma \setminus A} e^{G(x)}dx =\left(e^{-\frac{cB}{2}(1+o_n(1))} -\Prob_g(\Gamma \setminus A)\right) Z_g. $$

 By Lemma \ref{ggamma}, we have 
$$\Prob_g(\Gamma \setminus A) \leq \Prob_g(\Gamma_D^c)+\Prob_g \left ( \bigcup_{i \neq j}    \frac{|x_i-x_j|}{|i-j|_o} \geq  B     \right) \leq  2 e^{-c_1 n^3} +n^2 e^{-c_2 n \log^3 n}+  n^2 \max_{j,l}  \Prob_g \left( {\frac{|\xi_j^{(l)}|}{\sqrt{l}}}   \geq  B_{} \right).$$
Since the variance of $\xi_{j}^{(j)}$ is at most of the order $l$ by Proposition \ref{covofxi}, we have
\begin{equation} \label{c'}
\Prob_g \left( {\frac{|\xi_j^{(l)}|}{\sqrt{l}}}   \geq  B_{} \right) \leq e^{-c' B^2},
\end{equation}
where $c'$ is a universal positive constant.
Thus,
$$\Prob_g(\Gamma \setminus A) \leq 2 e^{-c_1 n^3} +n^2 e^{-c_2 n \log^3 n}+  n^2 e^{-c' B^2}   ,$$
and
$$\tilde{Z}_n \geq (e^{-\frac{cB}{2}(1+o_n(1))}-n^2 e^{-c' B^2}- 2 e^{-c_1 n^3}-n^2 e^{-c_2 n \log^3 n}) Z_g.$$
If $B \geq M \log^{\frac{1}{2}} n$ with some sufficient large $M>0$, then $n^2 e^{-c' B^2}+ 2 e^{-c_1 n^3}+n^2 e^{-c_2 n \log^3 n}$ is much smaller than $e^{-\frac{cB}{2}}$. We have
\begin{equation} \label{upper}
 \tilde{Z}_n \geq e^{-\frac{cB}{2}(1+o_n(1))} Z_g(1-o_n(1))  \geq Z_g e^{-cB} .
\end{equation}

\end{proof}

 Using the result of Lemma \ref{step1}, we can prove 
\begin{lemma}\label{step2}
For some sufficient large $M>0$, let $  M \log^{\frac{1}{2}} n \leq  B_k \leq D n^{\frac{1}{2}} \log^{\frac{3}{2}} n$. If there exists some $\gamma>0$ such that $k n^{-10} \leq n^{-\gamma}$ and
\begin{equation} \label{initial}
\Prob_x \left( x \in \Gamma_D: \max_{i \neq j} \frac{|x_i-x_j|}{|i-j|_o} \geq  B_k \right) \leq  k n^{-10},
\end{equation}
then
\begin{equation} \label{}
\Prob_x \left(x \in \Gamma_D: \max_{i \neq j} \frac{|x_i-x_j|}{|i-j|_o} \geq  B_{k+1} \right) \leq (k+1) n^{-10},
\end{equation}
with 
\begin{equation} \label{}
B_{k+1}=\sqrt{\frac{4c B_k+24 \log n}{c'}}.
\end{equation}
Here $c,c'$ are universal constants that do not depend on $n,k$.
\end{lemma}

\begin{proof} 
Denote $A_k=\{x \in \Gamma_D: \max_{i \neq j} \frac{|x_i-x_j|}{|i-j|_o} \leq  B_k\}$ and $A_{k}^{com} =\Gamma_D \setminus A_{k}$. \\
 Then $\Prob_x( A^{com}_k) \leq k n^{-10}$ by (\ref{initial}).
Since $B_{k+1} \leq B_{k}$, then $A_{k+1} \subset A_{k} \subset \Gamma_D$ and $  A^{com}_{k} \subset  A^{com}_{k+1}$. By Lemma \ref{step1}, if $x \in A_k$, then there exists a universal constant $c>0$, such that
$$\frac{Z_g}{\tilde{Z}_n} \leq e^{c B_k}.$$
Thus,
$$\Prob_x \left(A^{com}_{k+1}\right) =   \Prob_x(A_k^{com}) +\Prob_x(A_{k} \cap A_{k+1}^{com}) \leq k n^{-10} +\frac{1}{\tilde{Z}_n} \int_{A_{k+1}^{com}} e^{G(x)(1- \frac{c B_k}{n})} dx$$
\begin{equation}\label{J}
= k n^{-10}+ \frac{Z_g}{\tilde{Z}_n} \frac{1}{Z_g}  \int_{A_{k+1}^{com}} e^{G(x)(1- \frac{c B_k}{n})}dx \leq k n^{-10}+e^{cB_k} \left(\frac{1}{Z_g}  \int_{A_{k+1}^{com}} e^{G(x)(1- \frac{c B_k}{n})} dx \right) .
\end{equation}
Let $\tilde{x}=\sqrt{1-\frac{cB_k}{n}} x$. Then
$$ \frac{1}{Z_g}  \int_{A_{k+1}^{com}} e^{G(x)(1- \frac{c B_k}{n})} dx = \left( 1-\frac{cB_k}{n} \right)^{-\frac{n}{2}} \frac{1}{Z_g} \int_{\tilde{A}_{k+1}^{com}} e^{G(\tilde{x})} d \tilde{x} $$
\begin{equation}\label{K}
\leq e^{c B_k}  \frac{1}{Z_g} \int_{\tilde{A}_{k+1}^{com}} e^{G(\tilde{x})}  d \tilde{x}=e^{c B_k} \Prob_g(\tilde{A}^{com}_{k+1}),
\end{equation}
where ${\tilde{A}_{k+1}^{com}}=\left\{ x \in \Gamma_D:   \max_{i \neq j} \frac{|\tilde{x}_i-\tilde{x}_j|}{|i-j|_o} >B_{k+1} \sqrt{1-\frac{cB_k}{n}}  \right \} $.\\

 Note that $$\Prob_g(\tilde{A}_{k+1}^{com})  \leq \Prob_g \left( \max_{i \neq j} \frac{|x_i-x_j|}{|i-j|_o} \geq  B_{k+1} \sqrt{1-\frac{cB_k}{n}} \right) \leq n^2 \max_{j,l}  \Prob_g \left( {|\xi_j^{(l)}|}   \geq  B_{k+1} l \sqrt{1-\frac{cB_k}{n}} \right) $$
$$\leq n^2 \max_{j,l}  \Prob_g \left( {\frac{|\xi_j^{(l)}|}{\sqrt{l}}}   \geq  B_{k+1} \sqrt{1-\frac{cB_k}{n}} \right) \leq n^2 e^{-c' B_{k+1}^2 ({1-\frac{cB_k}{n}})},$$
where $c'>0$ is the universal constant introduced in (\ref{c'}). 
Thus
\begin{equation}\label{L}
\mbox{LHS of } (\ref{K})  \leq n^2 e^{-c' B_{k+1}^2(1-\frac{cB_k}{n})+{c B_k}} \leq n^2 e^{-\frac{c'}{2} B_{k+1}^2+c B_k} = e^{-\frac{c'}{2} B_{k+1}^2+c B_k +2 \log n}.
\end{equation}
Therefore, combining (\ref{J}), (\ref{K}), and (\ref{L}), we have
\begin{equation} \label{rhs}
\Prob_x \left(A^{com}_{k+1} \right)  \leq  k n^{-10}+e^{-\frac{c'}{2} B_{k+1}^2+2{c B_k}+2 \log n}.
\end{equation}
By letting the RHS of (\ref{rhs}) equal $(k+1) n^{-10}$,  we can solve $-c' B^2_{k+1}+4c B_k+4 \log n  = -20 \log n$ for $B_{k+1}$. We obtain $$ B_{k+1}=\sqrt{\frac{4c B_k+24 \log n}{c'}}.$$

\end{proof}

 Combining Lemma \ref{step1} and \ref{step2}, we have

\begin{proposition} \label{iteration}
There exist some constants $C_1, C_2>0$, such that for sufficient large $n$,
\begin{equation} \label{}
\Prob_x \left(  x \in \Gamma_D: \max_{i \neq j} \frac{|x_i-x_j|}{|i-j|_o} \geq  C_1 \log^{\frac{1}{2}}n  \right) \leq n^{-9},
\end{equation}
and
\begin{equation} \label{c1}
Z_g e^{-C_2 \log^{\frac{1}{2}} n} \leq \tilde{Z}_n \leq Z_g e^{{C_2} \log^{\frac{1}{2}} n}.
\end{equation}
\end{proposition}

\begin{proof}
 The definition of $\Gamma_D$ in (\ref{Gamma_o}) indicates that for $k=0$, (\ref{initial}) is satisfied with $B_0=D n^{\frac{1}{2}} \log^{\frac{3}{2}} n$. Then we use Lemma \ref{step1} and Lemma \ref{step2} to proceed the iteration by setting
\begin{equation} \label{iter}
B_{k+1}=\sqrt{\frac{4c B_k+24 \log n}{c'}}.
\end{equation}
Note that in Lemma \ref{step2}, $c, c'$ are universal constants.\\
  The fixed point of the iteration (\ref{iter}) is 
$$-c' B^2_{f}+4c B_f+4 \log n  = -20 \log n ~~~\Longrightarrow~~~ B_{f}=\frac{2c+ \sqrt{4c^2+24c' \log n}}{c'} \sim \log^{\frac{1}{2}} n.$$

 Recall that $B_0 \sim n^{\frac{1}{2}} \log^{\frac{2}{3}} n$. Moreover, if $B_k \geq C \log n$, then $B_{k+1} \sim \sqrt{B_k}$. This implies that for $B_k$ to reach the value of order $\log^{\frac{1}{2}} n$, one needs about $\log \log n$ iteration steps. In other words, the sequence $\{B_k\}$ will starts from $B_0=D n^{\frac{1}{2}} \log^{\frac{3}{2}} n$ and after $C_1' \log \log n$ number of iterations, $B_k$ decreases below the level $$B_\infty=C_1 \log^{\frac{1}{2}}n.$$

 Finally, we still need to check if the conditions of Lemma \ref{step1}, \ref{step2} are satisfied. To satisfy the first condition of Lemma \ref{step2} (also Lemma \ref{step1}), i.e. $M \log^{\frac{1}{2}} n \leq  B_k \leq D n^{\frac{1}{2}} \log^{\frac{3}{2}} n$, we need to modify the stopping time of the iteration process. We will end the iteration right before $B_k$ falls below $M \log^{\frac{1}{2}}n$. But the result remains the same. Note that the number of iteration steps is of the order $\log\log n$ , so the second condition of Lemma \ref{step2} also holds, i.e. $k n^{-10} \leq n^{-\gamma}$ for some $\gamma>0$.

 Therefore, we can find a subset, denoted as 
\begin{equation} \label{A}
A_{\infty}=\left\{ x \in \Gamma_D: \max_{i \neq j} \frac{|x_i-x_j|}{|i-j|_o} \leq  C_1 \log^{\frac{1}{2}}n \right\},
\end{equation}
such that
\begin{equation} \label{com12}
\Prob_x(\Gamma_D \setminus A_{\infty}) \leq n^{-9}.
\end{equation}
Moreover, by Lemma \ref{step1}, there exists some constant $C_2>0$, such that
$$Z_g e^{-C_2 \log^{\frac{1}{2}} n} \leq \tilde{Z}_n \leq Z_g e^{{C_2} \log^{\frac{1}{2}} n}.$$
\end{proof} 

 Combining (\ref{com12}) and (\ref{xgamma}), we have
\begin{corollary}\label{A_infty}
There exists a subset $A_{\infty}$ defined in (\ref{A}) such that
\begin{equation}\label{com1}
\Prob_x( A^c_{\infty}) \leq n^{-8}.
\end{equation}
\end{corollary}

\subsection{Step 2: Estimate of $F(x)$ under Gaussian distribution}

 If $x \in A_{\infty}$, which is defined in (\ref{A}),  then 
\begin{equation}\label{x12}
\max_{i \neq j} \frac{|x_i-x_j|}{|i-j|_o} \leq C_1 \log^{\frac{1}{2}}n,
\end{equation}
and by (\ref{F1}), we obtain that
\begin{equation} \label{gaussiantail4}
F(x) = O \left( \sum_{i > j} \frac{(x_i-x_j)^3}{n |i-j|_o^5} \right) =O \left( \log^{\frac{3}{2}} n \right).
\end{equation}
Note that 
\begin{equation} \label{gaussiantail2}
 \sum_{i > j, |i-j|_o > \log^3 n} \frac{|x_i-x_j|^3}{n |i-j|_o^5} \leq C \log^{\frac{3}{2}} n \sum_{ l > \log^3 n}  \frac{1}{ l^2} \leq C \log^{-\frac{3}{2}} n=o_n(1).
\end{equation}
Thus, it is sufficient for us to estimate
\begin{equation} \label{gaussiantail}
 \sum_{i > j; |i-j|_o \leq \log^3 n} \frac{(x_i-x_j)^3}{n |i-j|_o^5}= \sum_{k=1}^{\log^3 n} \frac{1}{l^5} \left( \frac{1}{n} \sum_{j=0}^{n-1} {\xi_j^{(l)} }^3\right),
\end{equation}
where $\xi^{(l)}_j$ is defined in (\ref{xi}). Denote 
\begin{equation}\label{omega_l}
{\Omega}_l=\left\{ \left| \frac{1}{n} \sum_{i=1}^{n-1} {\xi_i^{(l)}}^3 \right| \leq  n^{-\frac{1}{4}} \right\}, ~~1 \leq  l \leq \log^3 n,
\end{equation}
 and ${\Omega_{\infty}}=\bigcap_{l \leq \log^3 n} {\Omega}_l $. Note that if $x \in \Omega' \defeq \Omega_{\infty} \cap A_{\infty}$, then \begin{equation} \label{gaussiantail3}
 \sum_{l=1}^{\log^3 n} \frac{1}{l^5} \left| \frac{1}{n} \sum_{j=0}^{n-1} {\xi_j^{(l)} }^3\right| \leq  n^{-\frac{1}{4}} \sum_{l=1}^{\log^3 n} \frac{1}{l^5} =O( n^{-\frac{1}{4}}).
\end{equation} 

 Combining (\ref{gaussiantail}), (\ref{gaussiantail3}), and (\ref{gaussiantail2}), we have
\begin{equation} \label{tail}
F(x) =O( \log^{-\frac{3}{2}}n).
\end{equation}

 Using Proposition \ref{covofxi} and Lemma \ref{ggamma}, one can show that 
$$\Prob_g(A^c_{\infty}) \leq \Prob_g \left(  \max_{i \neq j} \frac{|x_i-x_j|}{|i-j|_o} \geq  C_1 \log^{\frac{1}{2}}n \right)+\Prob_g\left( \Gamma_D^c \right)$$
\begin{equation}\label{gcom1}
\leq n^2 e^{-C'_1 l  \log n }+2e^{-c_1 n^3}+n^2 e^{-c_2 n \log^3 n} \leq n^{-\frac{1}{4}},
\end{equation}
provided that $C_1$ and $n$ are chosen sufficiently large.

 Next, we want to show that $\Prob_g(\Omega_{\infty}^c)=o_n(1)$. The following lemma is useful.
\begin{lemma} \label{highmoment}
There exists some constant $C>0$ such that
\begin{equation} \label{}
\Prob_g\left(  \left|  \frac{1}{n} \sum_{i=1}^{n-1} {\xi_i^{(l)}}^3\right| > n^{-\frac{1}{4}} \right) \leq \frac{C l^4 \log n  }{n^{\frac{1}{2}}}.
\end{equation}
\end{lemma}

\begin{proof}
By Wick's formula, we have 
$$\E_g {\xi_j^{(l)}}^3 {\xi_k^{(l)}}^3=9 \left( \E_g {\xi_j^{(l)}}^2  \E_g{ \xi_k^{(l)}}^2 \E_g \xi_j^{(l)} \xi_k^{(l)} \right)+6 \left( \E_g \xi_j^{(l)} \xi_k^{(l)} \right)^3 .$$
By Proposition \ref{covofxi}, $\E_g{ \xi_k^{(l)}}^2 \leq C l$ and $ \E_g{ \xi_k^{(l)}} \E_g \xi_j^{(l)} \leq \frac{Cl^2}{|k-j|_o}$ for $|j-k|_o \geq l$. Then when $|j-k|_o \geq l$, 
$$\E_g {\xi_j^{(l)}}^3 {\xi_k^{(l)}}^3 \leq 9 C^3 \frac{l^4}{|j-k|_o} +6 C^3 \frac{l^6}{|j-k|_o^3} \leq  \frac{C' l^4}{|j-k|_o}.$$
Similarly,  if $|j-k|_o \leq l$, then $ \E_g{ \xi_k^{(l)}} \E_g \xi_j^{(l)} \leq C l$, and thus $\E_g {\xi_j^{(l)}}^3 {\xi_k^{(l)}}^3 \leq 15 C^3 l^3 \leq \frac{C' l^4}{|j-k|_o}.$
Therefore, there exists some constant $C'>0$ such that
$$\E_g \left( \sum_{i=1}^{n-1} {\xi_i^{(l)}}^3 \right)^2=n \E_g {\xi_k^{(l)}}^6+4 \sum_{ j \neq k} \E_g {\xi_k^{(l)}}^3 {\xi_j^{(l)}}^3  \leq C' l^4 n \log n.$$

 It follows from the Markov Inequality that

\begin{equation}\label{argument}
\Prob_g\left(  \left| \frac{1}{n} \sum_{i=1}^{n-1} {\xi_i^{(l)}}^3 \right|> n^{-\frac{1}{4}} \right) \leq \frac{\E_g \left( \sum_{i=1}^{n-1} {\xi^{(l)}}^3 \right)^2}{n^{\frac{3}{2}}}  \leq \frac{C' l^4 \log n  }{n^{\frac{1}{2}}}.
\end{equation}
\end{proof}

 By using Lemma \ref{highmoment}, it can be shown directly that, for sufficiently large $n$, we have
\begin{equation} \label{gcom2}
\Prob_g \left(  \Omega_{\infty}^c  \right) \leq \Prob_g \left( \bigcup_{l \leq \log^3 n} \Omega_{l}^c  \right)  \leq 
 \frac{C' l^4 \log^3 n  }{n^{\frac{1}{2}}} \leq \frac{C ' \log^{15} n  }{n^{\frac{1}{2}}} \leq n^{-\frac{1}{4}}.
\end{equation}
Let $x \in \Omega'=\Omega_{\infty} \cap A_{\infty}$. Combining (\ref{gcom1}) and (\ref{gcom2}), we have
\begin{equation} \label{gcom}
\Prob_g(\Omega'^c) \leq \Prob_g(\Omega_{\infty}^c) + \Prob_g(A_{\infty}^c) \leq 2 n^{-\frac{1}{4}}.
\end{equation}

 Therefore, we have the following lemma.
\begin{lemma}
There exists a subset of $\Omega' \subset \Gamma$
such that $$\Prob_g(\Omega'^c)=o_n(1),$$
and if $x \in \Omega'$, then $F(x)=o_n(1)$.

\end{lemma}

\subsection{Step 3:  Combining Step 1 and Step 2}

In this subsection, we finish the proofs of Proposition \ref{tech1} and Theorem \ref{tech2} by combining the results in Step 1 and Step 2. In Step 1, we have showed that $\Prob_x(A^c_{\infty})=o_n(1)$. In Step 2, we have obtained that $\Prob_g(\Omega^c_{\infty})=o_n(1)$ and $\Prob_g(A^c_{\infty})=o_n(1)$.

\begin{proof} (Proposition \ref{tech1} and Theorem \ref{tech2})

   We start by showing $\Prob_x \left( {\Omega}_{\infty}^c  \right)=o_n(1).$ Recall that ${\Omega}_l=\left\{ \left| \frac{1}{n} \sum_{i=1}^{n-1} {\xi_i^{(l)}}^3 \right| \leq  n^{-\frac{1}{4}} \right\} $,  $1 \leq l \leq \log^3 n$, and ${\Omega_{\infty}}=\bigcap_{l \leq \log^3 n} {\Omega}_l $. 
Then
$$\Prob_x \left( {\Omega}_l^c  \right) \leq \Prob_x(\Gamma_D^c)+\Prob_x \left( A_{\infty}^c  \right)+\Prob_x \left(  {\Omega}^c_l \cap A_{\infty} \cap \Gamma_D  \right).$$
By (\ref{xgamma}) and (\ref{com1}), 
\begin{equation}\label{}
 \Prob_x \left( {\Omega}_l^c  \right) \leq n^{-cn}+n^{-8}+\frac{1}{\tilde{Z}_n} \int_{{\Omega}_l^c} e^{G(x)(1-\frac{C_1 \log^{\frac{1}{2}}}{n})} dx \leq 2 n^{-8}+\frac{Z_g}{\tilde{Z}_n} \frac{1}{Z_g} \int_{{\Omega}_l^c} e^{G(x)(1-\frac{C_1 \log^{\frac{1}{2}}}{n})} dx.
\end{equation}
Changing the variable $\tilde{x}=\sqrt{1-\frac{C_1 \log^{\frac{1}{2}}n }{n}} x$, by Lemma \ref{scale}, we have
\begin{equation}\label{b1}
\Prob_x \left( {\Omega}_l^c  \right)  \leq 2 n^{-8}+ e^{C_1 \log^{\frac{1}{2}} n}   \cdot \Prob_g (\tilde{\Omega}_l^c) ,
\end{equation}
where $\tilde{\Omega}^c_l=\left\{ \left| \frac{1}{n} \sum_{i=1}^{n-1} \tilde{\xi_i^{(l)}}^3 \right| \geq  n^{-\frac{1}{4}} \left({1-\frac{C_1 \log^{\frac{1}{2}}n }{n}}\right)^{\frac{3}{2}}\right\}.$\\
Using the Markov inequality and following the arguments in (\ref{argument}) in Lemma \ref{highmoment},  we have 
\begin{equation} \label{a1}
\Prob_g (\tilde{\Omega}_l^c)=\frac{1}{Z_g} \int_{\tilde{\Omega}_l^c} e^{G(\tilde{x})} d\tilde{x} \leq \frac{C l^4 \log (n-l)  }{n^{\frac{1}{2}}\left({1-\frac{C_1 \log^{\frac{1}{2}}n }{n}}\right)^3} \leq \frac{C' l^4 \log n}{n^{\frac{1}{2}}}.
\end{equation}
Combining (\ref{a1}), (\ref{b1}), and (\ref{c1}), we obtain 
$$\Prob_x \left( {\Omega}_l^c  \right) \leq 2n^{-8}+\frac{C' l^4 \log n  e^{(C_1+C_2) \log^{\frac{1}{2}} n} }{n^{\frac{1}{2}} }.$$
Thus,
\begin{equation} \label{com2}
\Prob_x \left( \Omega_{\infty}^c \right) =\Prob_x\left(\bigcup_{l \leq \log^3 n} {\Omega}_l^c\right)  \leq 2n^{-8} \log^3 n +\frac{C'  \log^{13} n e^{C_3 \log^{\frac{1}{2}} n} }{n^{\frac{1}{2}} }=O(n^{-\frac{1}{4}}).
\end{equation}

 Repeating the arguments from Step 2, we have that (\ref{gaussiantail4}), (\ref{gaussiantail2}) and  (\ref{gaussiantail}) hold for $x \in A_{\infty} $ (recall that $A_{\infty}$ is defined in (\ref{A})).
If $x \in {\Omega_{\infty}}$, (\ref{gaussiantail3}) also holds.
Therefore,  if $ x \in \Omega'=A_{\infty} \cap \Omega_{\infty},$ the bound (\ref{tail}) on $F(x)$ still holds. In addition, by (\ref{com1}) and (\ref{com2}), we have \begin{equation} \label{com}
\Prob_x(\Omega'^c) \leq \Prob_x(\Omega_{\infty}^c)+\Prob_x(A_{\infty}^c) =O(n^{-\frac{1}{4}}).
\end{equation}
Thus, we have proved Proposition \ref{tech1}. Combining (\ref{com}), (\ref{gcom}), and (\ref{tail}), one can show that
$$(1-C' \log^{-\frac{3}{2}} n) Z_g \leq \tilde{Z}_n \leq (1+C' \log^{-\frac{3}{2}} n) Z_g.$$
If $A$ is a measurable subset of $ \Omega'$, then
$$\Prob_x(A)=\frac{1}{\tilde{Z_n}} \int_{A} e^{G(x)+F(x)} \leq (1+C'' \log^{-\frac{3}{2}} n) \frac{1}{Z_g} \int_A e^{G(x)} dx= \Prob_g(A)\left(1+O\left(\log^{-\frac{3}{2}} n\right)\right).$$
Similarly, we have $$\Prob_x(A) \geq (1-C'' \log^{-\frac{3}{2}} n) \frac{1}{Z_g} \int_A e^{G(x)} dx= \Prob_g(A)\left(1-O\left(\log^{-\frac{3}{2}} n\right)\right).$$
Combining with (\ref{com}) and (\ref{gcom}), we conclude that for any measurable set $A \subset \Gamma$, 
\begin{equation} \label{diff}
|\Prob_x(A)-\Prob_g(A)| =O\left(\log^{-\frac{3}{2}} n\right).
\end{equation}

\end{proof}

\section{Proof of Theorem \ref{weak}}
In this section, we prove the functional convergence in distribution of $\zeta_n(t)$ to $\zeta(t)$. 
\begin{proof}
First, we establish the convergence of finite-dimensional distributions. Fix finitely many $0 \leq t_1, \cdots, t_m \leq 2 \pi$. Let $j_l=\lfloor \frac{n t_l}{2 \pi} \rfloor, ~l=1, \cdots, m.$ Because of (\ref{x12}) and the construction of $\zeta_n(t)$, with probability $1-o_n(1)$, we have
$$\left|\zeta_n(t_l)-\frac{x_{j_l}}{\sqrt{n}}\right| \leq \left|\frac{x_{j_l}}{\sqrt{n}}-\frac{x_{j_l+1}}{\sqrt{n}}\right|=o_n(1).$$
Using the definition of $\zeta(t)$, one can also show that $\left|\zeta(t_l)-\zeta \left(\frac{2 \pi j_l}{n}\right) \right| =o_n(1)$ with high probability. It is sufficient to prove that $\frac{x_j}{\sqrt{n}}=\zeta \left(\frac{2 \pi j_l}{n}\right)+o_n(1)$ with probability $1-o_n(1)$. By Theorem \ref{tech2}, the finite-dimensional distribution of $(x_{j_1}, \cdots, x_{j_m})$ can be approximated by the finite-dimensional distribution of the Gaussian law defined in (\ref{pg}). Without loss of generality, assume that $n$ is odd. For even case, similar considerations hold. Using the representation (\ref{C}) for $x_j$, we have
$$\frac{x_j}{\sqrt{n}}=\frac{2}{n} \sum_{k=1}^{\frac{n-1}{2}} \left( \cos \left({\frac{2 \pi  j k}{n}}\right) \frac{n^2}{2 k(n-k)} X_k-\sin \left( {\frac{2 \pi  j k}{n}}\right) \frac{n^2}{2 k(n-k)} Y_k \right) $$
$$= \sum_{k=1}^{\frac{n-1}{2}} \left(  \frac{\cos {\frac{2 \pi  j k}{n}}}{k} X_k-\frac{\sin {\frac{2 \pi  j k}{n}}}{k} Y_k \right)+\sum_{k=1}^{\frac{n-1}{2}} \left(  \frac{\cos {\frac{2 \pi  j k}{n}}}{n-k} X_k-\frac{\sin {\frac{2 \pi  j k}{n}}}{n-k} Y_k \right)$$
\begin{equation} \label{x/n}
= \sum_{k=1}^{\infty} \left(  \frac{\cos {\frac{2 \pi  j k}{n}}}{k} X_k-\frac{\sin {\frac{2 \pi  j k}{n}}}{k} Y_k \right)+e_n,~~ \mbox{for}~~ 0 \leq j \leq n-1,
\end{equation}
where $\{X_k\}$ and $\{Y_k\}$ are i.i.d. real standard normal random variables.
Here $$e_n= \sum_{k=\frac{n+1}{2}}^{\infty} \left(  \frac{\cos {\frac{2 \pi  j k}{n}}}{k} X_k-\frac{\sin {\frac{2 \pi  j k}{n}}}{k} Y_k \right)+\sum_{k=1}^{\frac{n-1}{2}} \left(  \frac{\cos {\frac{2 \pi  j k}{n}}}{n-k} X_k-\frac{\sin {\frac{2 \pi  j k}{n}}}{n-k} Y_k \right)$$
is negligible because
$$\V_g (e_n)=\sum_{k=\frac{n+1}{2}}^{\infty} \frac{1}{k^2} +\sum_{k=1}^{\frac{n-1}{2}} \left(  \frac{1}{n-k} \right)^2=O\left(\frac{1}{n} \right).$$
Therefore, by (\ref{diff}), for $1 \leq l \leq m$, 
\begin{equation} \label{xxx}
\zeta_n\left(\frac{2 \pi j_l}{n}\right)=\frac{x_{j_l}}{\sqrt{n}}=\sum_{k=1}^{\infty} \left(  \frac{\cos {\frac{2 \pi  j_l k}{n}}}{k} X_k-\frac{\sin {\frac{2 \pi  j_l k}{n}}}{k} Y_k \right)+e_n=\zeta\left(\frac{2 \pi j_l}{n}\right)+e_n,
\end{equation}
where $\{X_k\}$, $\{Y_k\}$ are iid real Gaussian variables, $e_n$ is a random error term with $\V (e_n) =o_n(1).$ Therefore, one proves that $\zeta_n(t)$ converges in finite dimensional distribution to $\zeta(t)$.\\

 Now, we turn our attention to functional convergence. Note that the sequence of the distributions of $\zeta_n(t)$ gives a family of probability measures on the space $C[0, 2 \pi]$. Because of the finite-dimension distribution convergence, it is sufficient for us to show the tightness of the distribution sequence. A sequence of probability measures $\{ \Prob_n \}$ is tight if only if the following two conditions hold (\cite{weakconvergence}) :

\begin{enumerate}
\item For any small $\eta>0$, there exist corresponding $a$ and $n_0$, such that 
$$\Prob_n\left( f: |f(0)| \geq a \right)  \leq \eta,~ \mbox{for}~~ n \geq n_0. $$
\item For any small $\epsilon, \eta>0$, there exist corresponding $\delta_0$ and $n_0$, such that
$$\Prob_n \left( f: \omega_f(\delta_0) \geq \epsilon  \right) \leq \eta,~ \mbox{for}~~ n \geq n_0,$$
where
$$\omega_f(\delta)=\sup \{ |f(s)-f(t)|: 0 \leq s, t \leq 2 \pi, |s-t| < \delta \}.$$
\end{enumerate}

 To check the first condition, by (\ref{xxx}), we note that 
\begin{equation} \label{}
\zeta_n(0)=\frac{x_0}{\sqrt{n}}=\sum_{k=1}^{\infty} \frac{1}{k} X_k+e_n,
\end{equation}
and thus there exists sufficient large $n_0$ such that if $n>n_0$,
\begin{equation} \label{}
\V (\zeta_n(0)) =\sum_{k=1}^{\infty} \frac{1}{k^2}+o_n(1) \leq \frac{\pi^2}{3}.
\end{equation}
Then
\begin{equation} \label{}
\Prob_x(|\zeta_n(0)| \geq a) \leq \frac{\pi^2}{3 a^2},
\end{equation}
we choose $a=\sqrt{\frac{\pi^2}{ 3 \eta}}$.

 Next, to check the second condition, we need the following lemma.
\begin{lemma}\label{delta1}
There exist positive constants $c_{k} (1 \leq k \leq 6)$ such that
$$\Prob_x \left(  |\zeta_n(t)-\zeta_n(s)| \leq c_1|t-s|^{\frac{1}{20}}+c_2 n^{-\frac{1}{10}} \log n,~~ \forall t,s \in [0, 2\pi]: |t-s| \leq \delta \right) $$
$$> 1- (c_3 e^{-c_4 n^{\frac{4}{5}}}+c_5 \delta^{\frac{4}{5}}+c_6 \log^{-\frac{3}{2}} n).$$
\end{lemma}

 Assuming that Lemma \ref{delta1} is proved, we can finish the proof of Theorem \ref{weak} by choosing $\delta_0$ and $n_0$ such that
\begin{equation} \label{}
c_1 \delta_0^{\frac{1}{20}} \leq \frac{\epsilon}{2},~~~~c_2 n_0^{-\frac{1}{10}} \log n_0 \leq \frac{\epsilon}{2};
\end{equation}
\begin{equation} \label{}
c_3 e^{-c_4 n_0^{\frac{4}{5}}} \leq \frac{\eta}{3},~~~~~c_5 \delta_0^{\frac{4}{5}} \leq \frac{\eta}{3},~~~c_6 \log^{-\frac{3}{2}} n_0 \leq \frac{\eta}{3}.
\end{equation}
\end{proof}

 Finally, we need to prove Lemma \ref{delta1}.
\begin{proof}(Lemma \ref{delta1})
Because of (\ref{diff}), it is sufficient to prove that 
$$\Prob_g \left(  |\zeta_n(t)-\zeta_n(s)| \leq c_1|t-s|^{\frac{1}{20}}+c_2 n^{-\frac{1}{10}} \log n,~~ \forall t,s \in [0, 2\pi]: |t-s| \leq \delta \right)$$
$$ > 1- (c_3 e^{-c_4 n^{\frac{4}{5}}}+c_5 \delta^{\frac{4}{5}}).$$
Without loss of generality, assume that $s < t$.\\
 Case 1: Let us fix $C_0>0$ and assume that $|t-s|  \leq  \frac{C_0}{n}$. Then there exist $ i, j $ with $|i-j| \leq C_0+2$, such that 
\begin{equation} \label{}
s \in \left[\frac{2 \pi (i-1)}{n},\frac{2 \pi i}{n}\right], ~~~t \in \left[\frac{2 \pi j}{n},\frac{2 \pi (j+1)}{n}\right].
\end{equation}
By the triangle inequality, we have 
$$\Prob_g \left( \left| \zeta_n(t)-\zeta_n(s)\right| >\epsilon  \right) \leq \sum_{i \leq k \leq j+1} \Prob_g \left(  \left|\zeta_n\left(\frac{2 \pi (k-1)}{n}\right)-\zeta_n\left(\frac{2 \pi k}{n}\right) \right| >  \frac{\epsilon}{|i-j|+2} \right).$$
Note that $\zeta_n\left(\frac{2 \pi k}{n}\right)=\frac{x_k}{\sqrt{n}}$ and $\V (x_k-x_{k-1}) \sim 1$ for all $k$. Then for some positive constant $C_2, C_3$ depending on $C_0$, we have
$$\Prob_g \left( \left|\zeta_n\left(\frac{2 \pi (k-1)}{n}\right)-\zeta_n\left(\frac{2 \pi k}{n}\right) \right| \geq C_2 \epsilon  \right) =\Prob_g \left( \left| x_{k-1}-x_k  \right| \geq C_2 \epsilon \sqrt{n}  \right) \leq C_3 e^{-c' C^2_2 \epsilon^2 n}.$$
Thus 
\begin{equation} \label{}
\Prob_g \left( \left| \zeta_n(t)-\zeta_n(s)\right| >\epsilon  \right) \leq (C_0+2) C_3 e^{-c' C^2_2 \epsilon^2 n}.
\end{equation}
Furthermore, there exist positive constants $C_4, C_5$ which only depend on $C_0$, such that
\begin{equation} \label{}
\Prob_g \left( \exists s,t \in [0, 2\pi],  |t-s| \leq \frac{C_0}{n} \mbox{~~and~~} \left| \zeta_n(t)-\zeta_n(s)\right| >\epsilon \right) \leq n^2 C_0 C_3 e^{-c' C^2_2 \epsilon^2 n} \leq C_5 e^{-C_4 {\epsilon^2 n}}.
\end{equation}

 Case 2:
If $|t-s| \gg \frac{1}{n}$, we introduce a new partition of $[s,t]$. We start by dividing the interval $[0, 2 \pi]$ into $2^k$ disjoint subintervals
\begin{equation} \label{}
\Delta^{(k)}_ l=\left[\frac{2 \pi }{2^k} l,\frac{2 \pi }{2^k} (l+1)\right],~~l=0,1, \cdots, 2^k-1.
\end{equation}
There exists the smallest $k=k_0$, and related $l_0$, such that $\Delta^{(k_0)}_{l_0} \subset [s,t]$. Note that $k_0$ and $l_0$ are unique. Let $S_0=\Delta^{(k_0)}_{l_0}$. If $[s,t] \neq S_0$, then there exists the unique smallest $k_1>k_0$, and one or two values of $l_1$, such that $\Delta^{(k_1)}_{l_1} \subset [s,t] \setminus S_0$ (we could potentially add $\Delta^{(k_1)}_{l_1}$  on the left of $\Delta^{(k_0)}_{l_0}$ or add one on the right). If there is only one value of $l_1$, let $\Delta^{(k_1)}_{a_1}=\Delta^{(k_1)}_{l_1}$ and $\Delta^{(k_1)}_{b_1}=\emptyset$. If there are two values of $l_1$, let $a_1$ be the smallest of the two and $b_1$ the largest. Set $S_1=S_0 \cup \Delta^{(k_1)}_{b_1} \cup \Delta^{(k_1)}_{a_1}$. We continue this process. For each $m \geq 2$, find a unique smallest $k_m>k_{m-1}$, $\Delta^{(k_m)}_{a_m}, \Delta^{(k_m)}_{b_m} \subset [s,t] \setminus S_{m-1}$ (recall that $\Delta^{(k_m)}_{b_m}$ might be empty). Let $S_m=S_{m-1} \cup \Delta^{(k_m)}_{b_m} \cup \Delta^{(k_m)}_{a_m}$. We will stop at $m=r$, when either $[s,t]=S_r$ or the length of $S_{r+1} \defeq [s,t] \setminus S_r \leq \frac{C_0}{n}$, where $C_0>0$ is some fixed constant. Thus,
\begin{equation} \label{partition}
[s,t]=S_{0} \cup S_{r+1} \cup \left( \bigcup_{m=1}^r \Delta^{(k_m)}_{b_m} \right) \cup \left( \bigcup_{m=1}^{r} \Delta^{(k_m)}_{a_m} \right).
\end{equation}
Note that 
\begin{equation} \label{k_0}
\frac{1 }{2^{k_0}} \leq |s-t| \leq \frac{3 }{2^{k_0}}, ~~~\frac{1}{2^{k_r}} \leq \frac{C_0}{n} \leq \frac{1}{2^{k_r-1}},
\end{equation}
i.e. $k_0 \sim -\log |s-t|, k_r  \sim  \log n$.\\

 For any interval $I=[a,b]$, define $D_n(I) \defeq |\zeta_n(a)-\zeta_n(b)|$. For example,  $D_n(\Delta_{l}^{(k)}) \defeq \left|  \zeta_n \left( \frac{2 \pi }{2^k} (l+1)\right)-\zeta_n \left( \frac{2 \pi }{2^k} l \right) \right|$. Let $a_0=l_0$. Then by the triangle inequality, we have
$$D_n([s,t]) \leq  D_n(\Delta^{(k_0)}_{l_0}) +\sum_{m=1}^r D_n(\Delta_{a_m}^{(k_m)})+\sum_{m=1}^r D_n(\Delta_{b_m}^{(k_m)})+D_n(S_{r+1}) \leq 2 \sum_{m=0}^r D_n(\Delta_{a_m}^{(k_m)})+D_n(S_{r+1}).$$
Since $|S_{r+1}| \leq \frac{C_0}{n}$, using the result in Case 1 and letting $\epsilon=n^{-\frac{1}{10}}$, we have for some constants $C_1, C_2>0$,
\begin{equation} \label{remain}
\Prob_g \left( D_n(S_{r+1}) >n^{-\frac{1}{10}},~~\exists S_{r+1} \subset [0, 2 \pi]: |S_{r+1}| \leq \frac{C_0}{n}  \right) \leq C_1 e^{-C_2 n^{\frac{4}{5}} }.
\end{equation}
 To estimate $D_n(\Delta_{a_m}^{(k_m)})$, we need the following lemma.
\begin{lemma}\label{delta2}
Fix $s,t \in [0, 2 \pi]$. Then there exist some constants $C_1$, $C_2$, $C_4>0$, such that 
$$\Prob_g \left( \left| \zeta_n(t)-\zeta_n(s)\right| >|t-s|^{\frac{1}{20}}  +2 n^{-\frac{1}{10}}\right) \leq 2 C_1 e^{-C_2 n^{\frac{4}{5}}}+ C_4 |t-s|^{\frac{9}{5}}.$$
\end{lemma}
\begin{proof}(Lemma \ref{delta2})
 Fix $s,t$. Then there exist $i, j$, such that $$s \in \left[\frac{2 \pi (i-1)}{n},\frac{2 \pi i}{n}\right], ~~~t \in \left[\frac{2 \pi j}{n},\frac{2 \pi (j+1)}{n}\right].$$
Thus, we have
$$\Prob_g \left( \left| \zeta_n(t)-\zeta_n(s)\right| >|t-s|^{\frac{1}{20}}+2 n^{-\frac{1}{10}}  \right) \leq \Prob_g \left(  \left|\zeta_n\left(\frac{2 \pi (i-1)}{n}\right)-\zeta_n\left(\frac{2 \pi i}{n}\right) \right| > n^{-\frac{1}{10}}\right)$$
$$+ \Prob_g \left(\left|\zeta_n\left(\frac{2 \pi j}{n}\right)-\zeta_n\left(\frac{2 \pi (j+1)}{n}\right)\right| >n^{-\frac{1}{10}}\right) +\Prob_g \left( \left|\zeta_n\left(\frac{2 \pi i}{n}\right)-\zeta_n\left(\frac{2 \pi j}{n}\right)\right|  >  |t-s|^{\frac{1}{20}}\right).$$
Note that $\zeta_n\left(\frac{2 \pi k}{n}\right)=\frac{x_k}{\sqrt{n}}$, and $\V (x_k-x_l) \sim |k-l|$. Thus, we have
$$\Prob_g \left( \left|\zeta_n\left(\frac{2 \pi i}{n}\right)-\zeta_n\left(\frac{2 \pi j}{n}\right)\right|  >  |t-s|^{\frac{1}{20}}\right)  \leq  \frac{C_3 |i-j|^2}{ n^2 |t-s|^{\frac{1}{5}}} \leq C_4 |t-s|^{\frac{9}{5}},$$
where the last inequality comes from $\frac{2 \pi |i-j|}{n} \leq |t-s|$. Therefore,
$$\Prob_g \left( \left| \zeta_n(t)-\zeta_n(s)\right| >|t-s|^{\frac{1}{20}}  +2 n^{-\frac{1}{10}}\right) \leq 2 C_1 e^{-C_2 n^{\frac{4}{5}}}+ C_4 |t-s|^{\frac{9}{5}}.$$
\end{proof}

 By the result of Lemma \ref{delta2},  we have
\begin{equation} \label{union1}
\Prob_g \left(\bigcup_{l=0}^{2^k-1}    \left\{ D(\Delta_{l}^{(k)}) >\left(\frac{2 \pi}{2^k}\right)^{\frac{1}{20}}+2 n^{-\frac{1}{10}}  \right\}\right) \leq 2^k \left(2 C_1 e^{-C_2 n^{\frac{4}{5}}}+C_4 \left(\frac{2 \pi}{2^k}\right)^{\frac{9}{5}} \right)$$
$$\leq 2^{k+1} C_1 e^{-C_2 n^{\frac{4}{5}}}+\frac{C_5}{(2^k)^{\frac{4}{5}}}.
\end{equation}
Therefore,
\begin{equation} \label{union2}
\Prob_g \left(\bigcup_{k \geq k_0}^{k_r}  \bigcup_{l=0}^{2^k-1}    \left\{  D(\Delta_{l}^{(k)})>\left(\frac{2 \pi}{2^k}\right)^{\frac{1}{20}}+2 n^{-\frac{1}{10}}  \right\}\right) \leq  C_6 2^{k_r} e^{-C_2 n^{\frac{4}{5}}}+\frac{C_7}{(2^{k_0})^{\frac{4}{5}}}.
\end{equation}
Combining with (\ref{remain}), we have
$$\Prob_g \left(\bigcap_{k \geq k_0}^{k_r}  \bigcap_{l=0}^{2^k-1}    \left\{  D(\Delta_{l}^{(k)}) \leq \left(\frac{2 \pi}{2^k}\right)^{\frac{1}{20}}+2 n^{-\frac{1}{10}}  \right\} \bigcap \left\{D(S_{r+1}) \leq n^{-\frac{1}{10}} ~~\forall S_{r+1} \subset [0, 2 \pi]: |S_{r+1}| \leq \frac{C_0}{n} \right\} \right)$$
\begin{equation} \label{union3}
 \geq  1-\left((C_1+C_6) 2^{k_r} e^{-C_2 n^{\frac{4}{5}}}+\frac{C_7}{(2^{k_0})^{\frac{4}{5}}}\right).
\end{equation}
Due to $r \leq k_r$ and (\ref{partition}), then
$$\mbox{LHS of } (\ref{union3}) \leq \Prob_g \left(  |\zeta_n(s)-\zeta_n(t)| \leq  C_8 \left(\frac{1}{2^{k_0}} \right)^{\frac{1}{20}} +C_9 k_r n^{-\frac{1}{10}} ,~~\forall s,t \in [0, 2\pi]: |t-s| \leq \delta  \right).$$
Therefore, combining with (\ref{k_0}), we have
$$\Prob_g \left(|\zeta_n(s)-\zeta_n(t)| \leq c_1 |s-t|^{\frac{1}{20}}+c_2 n^{-\frac{1}{10}} \log n,~~\forall s,t \in [0, 2\pi]: |t-s| \leq \delta  \right)  \geq 1-( c_3 e^{-c_4 n^{\frac{4}{5}}}+c_5 \delta^{\frac{4}{5}}).$$

\end{proof}

\section{Proof of Corollary \ref{CLT}}
This section is devoted to the proof of the Gaussian fluctuation of $\sqrt{n} \sum_{j=0}^{n-1} g(\theta_j)$.
\begin{proof} 

Taking the Taylor expansions of $g(\theta)$ around $\theta_j=\frac{2 \pi j}{n}+\psi$, we have for $0 \leq \delta \leq 1$,
\begin{equation} \label{clttaylor1}
\sqrt{n} \sum_{j=0}^{n-1} g(\theta_j)=\sqrt{n} \sum_{j=0}^{n-1} g\left(\frac{2 \pi j}{n}+\psi\right)+\sqrt{n} \sum_{j=0}^{n-1} g'\left(\frac{2 \pi j}{n}+\psi+\delta \frac{x_j}{n^2}\right) \frac{x_j}{n^2}.
\end{equation}

Due to the assumption $\sum_{k=-\infty}^{\infty} |k|^{\frac{3}{2}} |c_k| < \infty$ where $c_k$ are complex Fourier coefficients, $g(x) =\sum_{k=-\infty}^{\infty} c_k e^{ikx}$ converges absolutely. Thus, the first term of (\ref{clttaylor1}) can be written as
\begin{equation}\label{666}
\sqrt{n} \sum_{j=0}^{n-1} g\left(\frac{2 \pi j}{n}+\psi \right)=n^{\frac{3}{2}} c_0+ \sqrt{n} \sum_{k=\pm \frac{n}{2}, \pm n \cdots} n c_k e^{i k \psi}  + \sum_{k \neq 0, \pm \frac{n}{2}, \pm n \cdots} c_k e^{i k \psi} \sum_{j=0}^{n-1} e^{i\frac{2 \pi k}{n}j}.
\end{equation}
Because $\sum_{j=0}^{n-1} e^{i\frac{2 \pi k}{n}j}=0$ when $k \neq 0, \pm \frac{n}{2}, \pm n \cdots$,  the last term of (\ref{666}) vanishes. One can estimate the second term of (\ref{666}) in the following.

\begin{equation}
|\sqrt{n} \sum_{k=\pm \frac{n}{2}, \pm n \cdots} n c_k e^{i k \psi} | \leq \sum_{k=\pm \frac{n}{2}, \pm n \cdots} n^{\frac{3}{2}} |c_k| \leq \sum_{k=\pm \frac{n}{2}, \pm n \cdots} (2 |k|)^{\frac{3}{2}} |c_k| \leq 2^{\frac{3}{2}} \sum_{|k| \geq \frac{n}{2}}^{\infty} |k|^{\frac{3}{2}} |c_k| = o_n(1).
\end{equation}

Therefore, we have shown that the first term of (\ref{clttaylor1}) is almost deterministic and
\begin{equation}\label{mean}
\sqrt{n} \sum_{j=0}^{n-1} g\left(\frac{2 \pi j}{n}+\psi \right)=n^{\frac{3}{2}} c_0+o_n(1).
\end{equation}

Because of (\ref{clttaylor1}) and (\ref{mean}), define
 \begin{equation}\label{s_n}
 \mathfrak{s}_n(x)=\left( \sqrt{n} \sum_{j=0}^{n-1} g ( \theta_j )-n^{\frac{3}{2}} c_0 \right)=\sqrt{n} \sum_{j=0}^{n-1} g'\left(\frac{2 \pi j}{n}+\psi+\delta \frac{x_j}{n^2} \right)  \frac{x_j}{n^2} +o_n(1).
 \end{equation}
 
Since $e^{i t \mathfrak{s}_n(x)}$ is a bounded function, then
$$\E \left( e^{i t \mathfrak{s}_n(x)} \right) \leq \frac{1}{\tilde{Z_n}} \int_{\Omega'} e^{i t \mathfrak{s}_n(x)} e^{G(x)+o_n(1)} dx+\Prob_x(\Omega'^c) \leq  \frac{1+o_n(1)}{1-o_n(1)}\frac{1}{Z_g} \int_{\Gamma} e^{i t \mathfrak{s}_n(x)} e^{G(x)} dx+\Prob_x(\Omega'^c)$$
$$\leq \E_g \left( e^{i t \mathfrak{s}_n(x)} \right) (1+o_n(1)).$$
Here $\E_g$ indicates taking expectation with respect to the multivariate Gaussian distribution defined in (\ref{pg}).
Similarly, we can show that 
$$\E \left( e^{i t \mathfrak{s}_n(x)} \right) \geq \frac{1-o_n(1)}{1+o_n(1)} \frac{1}{Z_g}  \int_{\Omega'} e^{i t \mathfrak{s}_n(x)} e^{G(x)} dx=\frac{1-o_n(1)}{1+o_n(1)} \frac{1}{Z_g} \int_{\Gamma} e^{i t \mathfrak{s}_n(x)} e^{G(x)} dx-\frac{1-o_n(1)}{1+o_n(1)} \Prob_g(\Omega'^c)$$
$$ \geq \E_g \left( e^{i t \mathfrak{s}_n(x)} \right) (1-o_n(1)).$$
Therefore, we obtain that $$\E \left( e^{i t \mathfrak{s}_n(x)} \right)=\E_g \left( e^{i t \mathfrak{s}_n(x)} \right)(1+o_n(1)).$$ Thus, it is sufficient for us to prove that $\mathfrak{s}_n(x)$ converges in distribution to $N(0, \sum_{k=1} ^{\infty} |c_{k}|^2)$ in the Gaussian case, i.e. 
$$\E_g \left( e^{i t \mathfrak{s}_n(x)} \right)=\exp \left( { -\frac{t^2}{2} \sum_{k=1}^{\infty} |c_k|^2} \right) (1+o_n(1)).$$
Then, from now on, we assume that $x$ follows the Gaussian distribution. Next, by (\ref{s_n}), we estimate the dominant part of $ \mathfrak{s}_n(x)$, i.e. $\sqrt{n} \sum_{j=0}^{n-1} g'\left(\frac{2 \pi j}{n}+\psi+\delta \frac{x_j}{n^2} \right)  \frac{x_j}{n^2} $.
 
The assumption  $\sum_{k=-\infty}^{\infty} |k|^{\frac{3}{2}} |c_k| < \infty$ implies that $g' \in C^{\frac{1}{2}}(\abc)$. In other words, there exists some constant $A$ such that $|g'(x)-g'(y)| \leq A|x-y|^{\frac{1}{2}}$ for all $x,y$. Since $\V (x_j) \sim n$, with high probability, $|x_j| \leq n^{\frac{1}{2}+\gamma} (\gamma>0)$ for all $0 \leq j \leq n-1$. Then we have
$$\left|g'\left(\frac{2 \pi j}{n}+\psi+\delta \frac{x_j}{n^2}\right)-g'\left(\frac{2 \pi j}{n}+\psi\right)\right| \leq A \delta^{\frac{1}{2}} \frac{x_j^{\frac{1}{2}}}{n} \leq A \delta^{\frac{1}{2}} \frac{\left(n^{\frac{1}{2}+\gamma}\right)^{\frac{1}{2}}}{n}=O(n^{-\frac{3}{4}+ \frac{1}{2} \gamma} ).$$
Thus,
\begin{equation}\label{111}
\left| \sqrt{n} \sum_{j=0}^{n-1} g'\left(\frac{2 \pi j}{n}+\psi+\delta \frac{x_j}{n^2}\right) \frac{x_j}{n^2}-\sqrt{n} \sum_{j=0}^{n-1} g'\left(\frac{2 \pi j}{n}+\psi \right) \frac{x_j}{n^2}   \right| =O(n^{-\frac{3}{4}+ \frac{3}{2} \gamma})=o_n(1),
\end{equation}
if we choose $\gamma < \frac{1}{2}$.

 Therefore, with high probability, the RHS of (\ref{s_n}) can be written as
\begin{equation} \label{}
\sqrt{n} \sum_{j=0}^{n-1} g'\left(\frac{2 \pi j}{n}+\psi+\delta \frac{x_j}{n^2}\right) \frac{x_j}{n^2}+o_n(1)=\sqrt{n} \sum_{j=0}^{n-1} g'\left(\frac{2 \pi j}{n}+\psi\right) \frac{x_j}{n^2}+o_n(1).
\end{equation}
Since $x$ follows the centered Gaussian distribution, we only need to prove the convergence of the variance. 
Without loss of generality, we can assume that $n$ is odd.  Using (\ref{x/n}), we have
$$\sqrt{n} \sum_{j=0}^{n-1} g'(\alpha_j) \frac{x_j}{n^2} =\frac{1}{n} \sum_{j=0}^{n-1} g'(\alpha_j) \left( \sum_{k=1}^{\frac{n-1}{3}} \left(  \frac{\cos {\frac{2 \pi  j k}{n}}}{k} X_k-\frac{\sin {\frac{2 \pi  j k}{n}}}{k} Y_k \right)+e_n \right)$$
\begin{equation} \label{}
= \sum_{k=1}^{\frac{n-1}{3}} \frac{1}{k} \left( \frac{1}{n}  \sum_{j=0}^{n-1} g'(\alpha_j)  \cos {\frac{2 \pi  j k}{n}} \right)X_k -\sum_{k=1}^{\frac{n-1}{3}} \frac{1}{k} \left( \frac{1}{n}  \sum_{j=0}^{n-1} g'(\alpha_j)  \sin {\frac{2 \pi  j k}{n}} \right)Y_k  +r_n,
\end{equation}
where $r_n$ is a random error term with $\V r_n =O \left( \frac{1}{n} \right)$. 
Since $\sum_{k=-\infty}^{\infty} |k|^{\frac{3}{2}} |c_k| < \infty$, $g'(x)=\sum_{l=-\infty}^{\infty}  i l c_l e^{i l x}$ is convergent absolutely. Thus, for $0 \leq k \leq \frac{n-1}{3}$, we have
 \begin{equation} \label{}
  \frac{1}{n}  \sum_{j=0}^{n-1} g'(\alpha_j)  \cos {\frac{2 \pi  j k}{n}}=\frac{1}{2n} \sum_{j=0}^{n-1} \sum_{l=-\infty}^{\infty} i l c_l e^{i  \frac{2 \pi l}{n}j+i l \psi} (e^{i \frac{2 \pi k}{n} j}+e^{-i \frac{2 \pi k}{n} j})
  \end{equation}
 $$=\frac{1}{2n} \sum_{l=-\infty}^{\infty} i l c_l e^{i l \psi} \sum_{j=0}^{n-1} \left( e^{i  \frac{2 \pi (l+k)}{n}j} +e^{i  \frac{2 \pi (l-k)}{n}j} \right) =O \left( \sum_{l= \pm n, \cdots } l c_l \right) +\frac{1}{2} \sum_{l \pm k=0, \pm \frac{n}{2}, \cdots} i l c_l e^{i l \psi}$$
 $$=o_n(1)+\frac{i k c_k e^{i k \psi}-i k c_{-k} e^{-i k \psi}}{2}=o_n(1)+\frac{1}{2 \pi} \int_{0}^{2\pi} g'(x+\psi) \cos k x dx$$
 $$=-\frac{k}{2 \pi}  \int_{0}^{2 \pi} g(x) \sin k x dx+o_n(1)=-\frac{k}{2} b_k+o_n(1),$$
 where $b_k= \frac{1}{\pi}\int_{0}^{2\pi} g(x) \sin k x dx $. Similarly, we have
 \begin{equation} \label{}
\frac{1}{n}  \sum_{j=0}^{n-1} g'(\alpha_j)  \cos {\frac{2 \pi  j k}{n}}=o_n(1)+\frac{ k c_k e^{ i k \psi}+ k c_{-k} e^{-i k \psi}}{2}=\frac{k}{2} a_k+o_n(1),
\end{equation}
where $a_k= \frac{1}{\pi}\int_{0}^{2\pi} g(x) \cos k x dx $.
Thus,
$$\V_g \left( \sqrt{n} \sum_{j=0}^{n-1} g'(\alpha_j) \frac{x_j}{n^2} \right)=\frac{1}{4} \sum_{k=1}^{\infty} (a_k^2+b_k^2)+o_n(1)=\sum_{k=1} ^{\infty} |c_{k}|^2+o_n(1).$$
We have finished the proof of Corollary \ref{CLT}.

\end{proof}

\section{Appendix}

 In this section, we will give the proof of the formulas in Lemma \ref{lemma}. 
 
 \begin{proof}
 We start with the proof of the formula (\ref{formula1}). Let $z_{k}=e^{i \frac{2 \pi k}{n}}~(0 \leq k \leq n-1)$. Then,
$$\sum_{k=1}^{n-1} \frac{1}{ \sin^2 \left( \frac{ \pi k}{n} \right)}=\sum_{k=1}^{n-1} \left( \frac{2 i}{e^{i \frac{\pi k}{n}}-e^{-i \frac{\pi k}{n}}} \right)^2=-4 \sum_{k=1}^n \frac{z_k}{ \left( z_k-1 \right)^2}.$$
Let $$F_1(z)=\frac{1}{(z-1)^2(z^n-1)}=\frac{1}{(z-1)^3(1+z+\cdots+z^{n-1})},$$
which is a holomorphic function on $\C$ except $z=z_k~(0 \leq k \leq n-1)$. Note that $z=z_k ~(1 \leq k \leq n-1)$ are the simple poles of $F_1(z)$, and $$z^n-1=\prod_{j=0}^{n-1} (z-z_j), ~~~\prod_{j: j \neq k} (z_k-z_j)=n z_k^{-1}.$$
Combining it with the residue formula, we have $$\sum_{k=1}^{n-1} \frac{1}{ \sin^2 \left( \frac{ \pi k}{n} \right)}=-4 n \sum_{k=1}^{n-1} \Res_{z=z_k} F_1(z)=4n \Res_{z=1} F_1(z).$$
Since $z=1$ is a pole of order $3$, we obtain that
$$\Res_{z=1} F_1(z)=\frac{1}{2!} \left( \frac{1}{1+z+ \cdots +z^{n-1}} \right)''\Big|_{z=1}=\frac{n^2-1}{12 n}.$$
Thus we have proved the formula (\ref{formula1}).\\
 Similarly, we can prove (\ref{formula2}). Let $$F_2(z)=\frac{z}{(z-1)^4(z^n-1)}=\frac{1}{(z-1)^3(z^n-1)}+\frac{1}{(z-1)^4(z^n-1)}.$$
Then
$$\sum_{k=1}^{n-1} \frac{1}{ \sin^4 \left( \frac{ \pi k}{n} \right)}=16 \sum_{k=1}^{n-1} \frac{z_k^2}{(z_k-1)^4}=16 n \sum_{k=1}^{n-1} \Res_{z=z_k} F_2(z)=-16 n \Res_{z=1} F_2(z)$$
$$=\frac{(n^2-1)(n^2+11)}{45}.$$

 For the formula (\ref{formula3}), let $$F_3(z)=\frac{ (z^m-1)^2}{z^m(z-1)^2 (z^n-1) }=\frac{(1+z+\cdots +z^{m-1})^2}{z^m (z-1) (1+z+\cdots+z^{n-1})}.$$
$F_3$ is holomorphic on $\C$ except at $z=0$ and $z=z_k~(0 \leq k \leq n-1)$.
Similarly to the previous computations, we have
$$\sum_{k=1}^{n-1} \frac{\sin^2 \left(m  \frac{ \pi k  }{n}  \right)}{ \sin^2 \left( \frac{ \pi k}{n} \right)}=\sum_{k=1}^{n-1} \left( \frac{z_k^m-1}{z_k-1} \right)^2 z_k^{1-m}=-n \left(\Res_{z=1} F_3(z)+ \Res_{z=0} F_3(z)\right).$$
Since $z=1$ is a simple pole,
$$\Res_{z=1} F_3(z)=\frac{(1+z+\cdots +z^{m-1})^2}{z^m (1+z+\cdots+z^{n-1})} \Big|_{z=1}=\frac{m^2}{n}.$$
Also note that 
$$F_3(z)=-z^{-m} (z^{2m}-2z^m+1) \left( \sum_{k=0}^{\infty} z^k \right)^2 \left(\sum_{k=0}^{\infty}z^{kn} \right).$$
Since $1 \leq m \leq n-1$, 
$$\Res_{z=0} F_3(z)=-\Res_{z=0} \frac{1}{(z-1)^2 z^m}=-m.$$
Thus, we finish the proof of (\ref{formula3}).

 Finally, for (\ref{formula4}), let
$$F_4(z)=\frac{(z^m-1)^2}{(z-1)^4 (z^n-1)} \cdot z^{1-m}.$$
Similar to the proof of the formula (\ref{formula3}), we obtain that
$$\sum_{k=1}^{n-1} \frac{\sin^2 \left(m  \frac{ \pi k  }{n}  \right)}{ \sin^4 \left( \frac{ \pi k}{n} \right)}=-4 \sum_{k=1}^{n-1}  \frac{(z_k^m-1)^2}{(z_k-1)^4}  z_k^{2-m}=4n \left(\Res_{z=1} F_4(z)+ \Res_{z=0} F_4(z)\right)$$
$$=\frac{m^2(n-m)^2}{3}+\frac{2}{3} m (n-m).$$

\end{proof}

\end{document}